\documentclass{article}

\usepackage{amsmath,amsfonts,amssymb,amsthm}
\usepackage[breaklinks=true,hidelinks]{hyperref}

\usepackage{float}
\usepackage{color}
\usepackage{tikz-cd}
\usepackage[T1]{fontenc}
\usepackage[utf8x]{inputenc}

\usepackage{outlines}
\usepackage{enumitem}
\setenumerate[1]{label=\Roman*.}
\setenumerate[2]{label=\Alph*.}
\setenumerate[3]{label=\roman*.}
\setenumerate[4]{label=\alph*.}

\usepackage[normalem]{ulem}

\newsavebox{\pb} % The pullback symbol.
\sbox\pb{
\begin{tikzpicture}
\draw (0,0) -- (1ex,0ex);
\draw (1ex,0ex) -- (1ex,1ex);
\end{tikzpicture}}

\newcommand{\pullback}{\arrow[dr, phantom, "\usebox\pb" , very near start, color=black]}

\newcommand{\maps}{\colon}    %correct symbol for colon in f: X -> Y
\newcommand{\R}{{\mathbb R}}  %real numbers
\newcommand{\C}{{\mathbb C}}  %complex numbers
\newcommand{\Z}{{\mathbb Z}}  %integers

\newcommand{\G}{\mathcal{G}} % the bundle gerbe G.
\renewcommand{\H}{\mathcal{H}} % another bundle gerbe H
\newcommand{\I}{\mathcal{I}} % the trivial bundle gerbe I. 

\newcommand{\Oh}{\mathcal{O}} %the structure sheaf
\newcommand{\E}{\mathcal{E}} %sections of a vector bundle E
\newcommand{\F}{\mathcal{F}} %sections of a vector bundle F
\newcommand{\J}{\mathcal{J}} %the nilpotent ideal
\newcommand{\Cal}{\mathcal{C}} %the sheaf of smooth functions
\renewcommand{\L}{\mathcal{L}} %sections of a line bundle L
\newcommand{\T}{\mathcal{T}} %the tangent sheaf
\newcommand{\DelCpx}{\mathcal{D}} % Write the kth Deligne complex as \DelCpx(k).
\newcommand{\CDR}{\Omega_{\mathbb{C}}} % The complexified de Rham complex.
\newcommand{\ECDR}{\Omega_{\mathbb{C},0}} % The even complexified de Rham complex.
\newcommand{\Soul}{\Omega_s} % The soul of the de Rham complex.

\newcommand{\Grb}[2]{\mathbb{G}^\nabla_{ #2 }\left( #1 \right)} % Write \Grb{M}{H} for the gerbes on M with curvature H.

\newcommand{\U}{\mathcal{U}} % A good cover.

\newcommand{\DD}{\operatorname{DD}} %taking the Dixmier--Douady class
\newcommand{\Del}{\operatorname{Del}} %taking the Deligne cohomology class
\newcommand{\curv}{\operatorname{curv}} %taking the curvature

\renewcommand{\b}{{\mathfrak{b}}}  %shifting

\newcommand{\Hom}{{\rm Hom}} %Homomorphisms
\renewcommand{\hom}{\underline{{\rm Hom}}} %Internal hom
\newcommand{\inclusion}{\hookrightarrow}
\newcommand{\iso}{\cong} %isomorphism
\newcommand{\ns}{\phantom{-1}} % ns = null superscript, an insane person's attempt to beautify some equations

\newcommand{\define}[1]{{\bf \boldmath{#1}}}
\newcommand{\arxiv}[1]{\href{http://arxiv.org/abs/#1}{arXiv:{#1}}}

\newcommand{\Cech}{\v{C}ech}

\newtheorem{thm}{Theorem}    
\newtheorem{cor}[thm]{Corollary}
\newtheorem{lem}[thm]{Lemma}
\newtheorem{prop}[thm]{Proposition}

\theoremstyle{definition}

\newtheorem{defn}[thm]{Definition}
\newtheorem{ex}[thm]{Example}

\title{Bundle gerbes on supermanifolds}

\author{John Huerta\\[.5em]
{\small Department of Mathematics} \\[-.3em]
{\small Instituto Superior T\'ecnico} \\[-.3em]
{\small Av.\ Ravisco Pais 1} \\[-.3em]
{\small University of Lisbon} \\[-.3em]
{\small 1049-001 Lisbon, Portugal} \\
\small john.huerta@tecnico.ulisboa.pt
}

\begin{document}
\maketitle

\begin{abstract}
  We show that every bundle gerbe on a supermanifold decomposes into a bundle gerbe
  over the underlying manifold and a 2-form on the supermanifold. This decomposition
  is not canonical, but is determined by the choice of a projection map to the
  underlying manifold. Along the way, we prove the familiar cohomological
  classification theorems for bundle gerbes and bundle gerbes with connection in the
  context of supermanifolds.
\end{abstract}

\section{Introduction}

A `bundle gerbe' is a geometric object defined over a base manifold whose
characteristic class sits in the third cohomology of the base \cite{Murray}. It is
analogous to a complex line bundle, and its characteristic class is an analogue of
the Chern class in degree 2. A line bundle can be equipped with a connection, and
when this is done the curvature of the connection is a closed 2-form and is a de Rham
representative of the Chern class. Analogously, a bundle gerbe can also be equipped
with a connection, and when this is done the curvature is a closed 3-form and is a de
Rham representative of the characteristic class.

The theory of line bundles with connection has been extended to supermanifolds by
Kostant \cite{Kostant} in his study of prequantization in supergeometry. We extend
this analysis to bundle gerbes on supermanifolds, proving an analogous cohomological
classification to the one on ordinary manifolds: bundle gerbes with connection are
classified, up to equivalence, by the Deligne cohomology of a supermanifold.

We then analyze exactly what is gained by extending the theory of bundle gerbes from
manifolds to supermanifolds: up to equivalence, every bundle gerbe on a supermanifold
is determined by a bundle gerbe on the underlying manifold, called the body, and a
2-form on the supermanifold. This decomposition is not canonical, but depends on a
choice of projection to the body.

To state this result explicitly, let $M$ be a supermanifold. For any supermanifold,
there is a canonical inclusion of an ordinary manifold $i \maps M_b \inclusion
M$. Invoking the poetic terminology of DeWitt~\cite{DeWitt}, we call $M_b$ the
\define{body} of $M$, though it is more commonly called the \define{reduced
manifold}. A bundle gerbe $\G$ on $M$ can be pulled back along the inclusion $i$ to
yield a bundle gerbe $i^* \G$ on $M_b$. It turns out every bundle gerbe on $M_b$ can be
obtained in this way, up to equivalence.

What about conversely? Can we recover $\G$ from $i^* \G$? To examine this, we choose
a \define{body projection}: a map $p \maps M \to M_b$ satisfying $pi =
1_{M_b}$. Unlike the inclusion $i$, this map is not canonical, but it does exist in
the category of smooth supermanifolds. Our main result is that every bundle gerbe
$\G$ on the supermanifold $M$ decomposes as follows, up to equivalence:
\[ \G \simeq p^* \G_b \otimes \I_\beta . \]
Here, $\G_b$ is bundle gerbe on the body $M_b$, and $\I_\beta$ is the trivial bundle
gerbe with `curving' given by the 2-form $\beta \in \Omega^2(M)$. Tensoring with
$\I_\beta$ has the effect of adding the 2-form $\beta$ to the curving of $p^* \G_b$,
a part of the connection data. Moreover, $\G_b$ is unique---it must be equivalent to
$i^* \G$---whereas $\beta$ is unique up to the addition of an exact 2-form, provided
it is chosen so that $i^* \beta = 0$, as can always be done.

With enough experience in supergeometry, this decomposition makes intuitive
sense. Folklore tells us that all the topological information in $M$ lies in the body
$M_b$. Bundle gerbes without connection are topological objects, so we expect a
bijection between bundle gerbes without connection on $M_b$ and those on $M$. Thus
any bundle gerbe $\G$ on $M$ without connection must satisfy $\G \simeq p^* \G_b$ for
some bundle gerbe $\G_b$ on $M_b$ without connection. Tensoring with the dual of
$\G$, this says:
\[ \G^* \otimes p^* \G_b \simeq \I ,\]
where $\I$ is the trivial bundle gerbe without connection. Equipping $\G$ and $\G_b$
with connection, we thus have:
\[ \G^* \otimes p^* \G_b \simeq \I_\beta ,\]
because a trivial bundle gerbe with connection $\I_\beta$ is a 2-form, $\beta$. This
is essentially our result.

Our work is not the first to consider bundle gerbes on supermanifolds; there are two
important antecedents. First, bundle gerbes are a special case of principal 2-bundles
\cite{Bartels, NikolausWaldorf}, and Schreiber has built the machinery for principal
$\infty$-bundles over supermanifolds~\cite{Schreiber}. For bundle gerbes, he
constructs a higher stack denoted $\mathbf{B}^2{\rm U}(1)_{\rm conn}$ on the site of
supermanifolds, and the objects of this stack are bundle gerbes with
connection. Separately, Suszek has constructed bundle gerbes over supermanifolds of
relevance to string theory, focusing on the subtle interplay of bundle gerbes with
supersymmetry, $\kappa$-symmetry, and \.Inönü--Wigner contraction~\cite{Suszek}.

This paper is organized as follows: in Section \ref{sec:supergeometry}, we outline
the basics of supergeometry. Section \ref{sec:gerbes} is the core of the paper. In
Section \ref{sec:noconn}, we define bundle gerbes over supermanifolds and show they
are classified, up to equivalence, by the Dixmier--Douady class in degree 3. In
Section \ref{sec:conn}, we add connection data to our bundle gerbes, and show that
bundle gerbes with connection are classified by Deligne cohomology. In Section
\ref{sec:curvature}, we define the curvature 3-form, and give necessary and
sufficient conditions for a 3-form to be the curvature of a bundle gerbe. Finally, in
Section \ref{sec:bodyandsoul}, we prove our main result, decomposing a bundle gerbe
on a supermanifold into a bundle gerbe on the body and a 2-form on the supermanifold.

\subsection*{Acknowledgments}

We thank Roger Picken for countless conversations, and Michael Murray and Danny
Stevenson for helping with particular issues. We thank Kowshik Bettadapura and
Rafa\l{} Suszek for a close reading of a previous version of this article, and Dan
Berwick-Evans for a stimulating suggestion. The last conversation occurred at the
Perimeter Institute during the wonderful workshop \define{QFT for Mathematicians} in
June 2019. We thank Perimeter for their hospitality. This work was partially
supported by the FCT project grant \define{Higher Structures and Applications},
PTDC/MAT-PUR/31089/2017.

\section{Elements of supergeometry}
\label{sec:supergeometry}

In the physics literature, a `supermanifold', often called a `superspace', is a space
with both even, commuting coordinates and odd, anticommuting coordinates. For
mathematicians, it is familiar that any point of an $n$-dimensional manifold has a
neighborhood that we can describe using $n$ coordinate functions, $\{x_i\}$. Since
these are real-valued functions, they pairwise commute in the algebra of smooth
functions on this neighborhood. The idea of the generalization from manifold to
supermanifold is that any point of a $m|n$-dimensional supermanifold has a
neighborhood we can describe using $m$ even coordinates $\{x_i\}$, which
pairwise commute, along with $n$ odd coordinates $\{\theta_j\}$, which
pairwise anticommute:
\[ \theta_j \theta_{j'} = - \theta_{j'} \theta_j . \]
Since real-valued functions form a commutative algebra, the coordinates $\{ x_i,
\theta_j \}$ {\em are not functions}. Instead, they are elements of a $\Z_2$-graded
commutative algebra, which is related to the algebra of smooth functions in a precise
way: we obtain the algebra of smooth functions by quotienting out by the ideal
generated by $\{\theta_j\}$, or in other words by setting $\theta_j = 0$. This is
the beginning of supergeometry.

In this section, we give a rapid review of those parts of supergeometry we will
need. For a more leisurely introduction, we recommend the survey article of
Leites~\cite{Leites}, the books of Manin~\cite{Manin} or Caston, Carmeli and
Fioresi~\cite{CCF}, or the notes of Deligne and Morgan~\cite{DM}.

\subsection{Superalgebra}

We begin with the algebra underlying supergeometry. A \define{super vector space} $V$
is a $\Z_2$-graded real vector space, $V = V_0 \oplus V_1$. We refer to the
$\Z_2$-grading as the \define{parity}, and say elements of $V_0$ are \define{even}
while those of $V_1$ are \define{odd}. As usual, elements in $V_i$ are called
\define{homogeneous}, and the parity of a homogeneous element $v \in V_i$ is denoted
by vertical braces: $|v| = i$.

A \define{supercommutative superalgebra} $A$ is a $\Z_2$-graded real associative
algebra with unit, $A = A_0 \oplus A_1$, such that $ab = (-1)^{|a||b|} ba$ for all
homogeneous elements $a$ and $b$ of parity $|a|$ and $|b|$, respectively. The key
example of a supercommutative superalgebra is the \define{Grassmann algebra}
$\Lambda(\theta_1, \ldots, \theta_n)$, free on $n$ odd generators $\theta_i$.

A \define{left module} of a supercommutative superalgebra $A$ is super vector space
$\mathcal{M}$ where $A$ acts on the left, $A \otimes \mathcal{M} \to \mathcal{M}$,
respecting the unit, multiplication, and parity:
\begin{itemize}
\item $1m = m$;
\item $a(bm) = (ab)m$;
\item $|am| = |a| + |m|$.
\end{itemize}
A \define{right module} is defined similarly, and the \define{tensor product}
$\mathcal{M} \otimes_A \mathcal{N}$ of left module $\mathcal{M}$ with a right module
$\mathcal{N}$ is the usual notion of graded tensor product. Because $A$ is
supercommutative, any left module is automatically a right module, and vice versa,
when we define $am = (-1)^{|a||m|} ma$ for $a \in A$ and $m \in \mathcal{M}$. Thus we
will stop distinguishing left and right modules.

We say that a module $\mathcal{M}$ is \define{free of finite rank} if it admits a
finite homogeneous basis: this is a pair of finite subsets $\{e_1, \ldots, e_m\}
\subseteq \mathcal{M}_0$ and $\{f_1, \ldots, f_n\} \subseteq \mathcal{M}_1$ such that
the union is linearly independent over $A$ and spans $\mathcal{M}$. The pair of
natural numbers $m|n$ is called the \define{rank} of $\mathcal{M}$ and is independent
of the choice of basis. In supergeometry, the rank or dimension is always an ordered
pair of natural numbers, and we adopt the notation $m|n$ for this ordered pair.

\subsection{Supermanifolds}
\label{sec:superman}

Having described the necessary algebra, we are ready to introduce the main objects of
study in supergeometry.
\begin{defn}

A \define{supermanifold} $M$ of dimension $m|n$ is a pair $(M_b, \Oh_M)$, where:
\begin{itemize}
\item $M_b$ is a topological manifold of dimension $m$, called the \define{body} of
  $M$; in particular, $M_b$ is paracompact, second countable and Hausdorff;
\item $\Oh_M$ is a sheaf of supercommutative superalgebras on $M_b$, called the
  \define{structure sheaf} of $M$; the sheaf $\Oh_M$ is required to be
  \define{local}, meaning that each stalk $\Oh_{M,x}$ is a local ring, containing a
  unique maximal ideal;
\item $M$ is equipped with an atlas extending a topological atlas of $M_b$: for each
  chart $(U, \varphi \maps U \to \R^m)$ in our atlas, we choose a local isomorphism of
  sheaves of superalgebras:
  \[ \phi \maps \Cal^\infty_{\varphi(U)} \otimes \Lambda(\theta_1, \ldots, \theta_n)
    \to \Oh_M|_U . \]
  Here, $\Cal^\infty_{\varphi(U)}$ is the sheaf of smooth functions on the open set
  $\varphi(U) \subseteq \R^m$, $\Oh_M|_U$ is the restriction of $\Oh_M$ to the open
  set $U \subseteq M_b$, and $\phi$ is \define{local} in the sense that it preserves
  the maximal ideals on the stalks. The triple $(U, \varphi, \phi)$ is called a
  \define{chart} on $M$, and the collection of charts is an \define{atlas} for $M$.
\end{itemize}

\end{defn}

\begin{ex}
  Here are some first examples of supermanifolds:
  \begin{itemize}
  \item Any smooth manifold $X$ of dimension $m$ is trivially a supermanifold of
    dimension $m|0$, with structure sheaf the sheaf of smooth functions
    $\Cal^\infty_X$.
  \item A \define{super Cartesian space} is a supermanifold of the form $\R^{m|n} =
    (\R^m, \Cal^\infty_{\R^m} \otimes \Lambda(\theta_1, \ldots, \theta_n))$.
  \item A \define{super domain} $U^{m|n}$ of dimension $m|n$ is the restriction of
    the structure sheaf on $\R^{m|n}$ to an open set $U \subseteq \R^{m}$; in other
    words, $U^{m|n} = (U, \Cal^\infty_U \otimes \Lambda(\theta_1, \ldots,
    \theta_n))$.
  \item More generally, given any supermanifold $M$ and open set $U \subseteq M_b$ of
    the body, $U$ becomes a supermanifold in the obvious way: $U = (U, \Oh_M|_U)$.
  \item Given any vector bundle $E \to X$ over a smooth manifold $X$, we obtain the
    supermanifold denoted $\Pi E = (X, \Lambda \E^*))$, where we write
    $\E$ for the sheaf of sections of $E$. In words, the structure sheaf of $\Pi E$
    is given by sections of the exterior algebra bundle of $E^*$, the dual of $E$. If
    $E$ has rank $n$ and $X$ has dimension $m$, then $\Pi E$ is a supermanifold of
    dimension $m|n$.
  \item In particular, for the tangent bundle $TX \to X$, we get the \define{odd
    tangent bundle} $\Pi TX = (X, \Omega_X^\bullet)$, whose structure sheaf is the
    algebra of differential forms, regarded as a supercommutative superalgebra.
  \end{itemize}
\end{ex}

We define a \define{smooth map} $f \maps M \to N$ between supermanifolds to be a pair $(f,
f^*)$, where $f \maps M_b \to N_b$ is continuous, and $f^* \maps \Oh_N \to f_* \Oh_M$
is a homomorphism of sheaves of superalgebras over $N_b$. Here $f_* \Oh_M$ denotes
\define{pushforward} of $\Oh_M$ along $f$: for any open $V \subseteq N_b$, we define
$f_* \Oh_M(V) := \Oh_M(f^{-1}(V))$.

A key technical role will be played by submersions and fiber products of
supermanifolds. A map of supermanifolds $\pi \maps Y \to M$ is a \define{submersion}
if it is locally isomorphic to a projection map $p_1 \maps U \times V \to U$, where
$U$ and $V$ are super domains. It is immediate from this definition that any
submersion admits local sections. With a bit of work, we can show that pullbacks
along submersions exist in the category of supermanifolds. In particular, given a
submersion $\pi \maps Y \to M$, the \define{fiber square} $Y^{[2]} = Y \times_M Y$ is
defined to be the pullback of $\pi$ along itself:
\[
  \begin{tikzcd}
    Y \times_M Y \ar[r] \ar[d] \pullback  & Y \ar[d, "\pi"] \\
    Y \ar[r, "\pi"] & M
  \end{tikzcd}
\]
In fact, any \define{fiber power} $Y^{[p]}$ of $Y$ exists, and is defined to be the
fiber product $Y \times_M \cdots \times_M Y$ of $Y$ with itself $p$ times. Between
the fiber powers, we have the maps $\pi_i \maps Y^{[p+1]} \to Y^{[p]}$, the
projection that omits that $i$th factor. More generally, we have the maps
$\varpi_{i_1, i_2, \ldots, i_q} \maps Y^{[p]} \to Y^{[q]}$ that project onto the
$i_1$th, $i_2$th, \ldots, $i_q$th factors, in that order, for any $1 \leq i_1 < i_2 <
\cdots < i_q \leq p$.

\subsection{Bundles in supergeometry}

Line bundles on $M$ will play a central role in our story. As in algebraic geometry,
bundles in supergeometry correspond to modules: a \define{vector bundle} on a
supermanifold $M$ is a sheaf of $\Oh_M$-modules, locally free of finite rank. In
particular, a \define{complex line bundle} on $M$, called simply a \define{line
bundle} for short, is a sheaf $\L$ of modules of the complexified structure sheaf,
$\Oh_M \otimes \C$, locally free of rank $1|0$. This last condition means that any
point $x \in M_b$ lies in a neighborhood $U$ on which there is an even section $s \in
\L(U)$ such that $\L|_U = \Oh_\C \cdot s$. Here, we write $\Oh_\C$ for the
complexified structure sheaf $\Oh_M \otimes \C$, a notation we will use whenever $M$
is clear from the context. We call the section $s$ a \define{trivializing section}.

Given two vector bundles $\E$ and $\F$ on $M$, a \define{bundle map} $\varphi \maps
\E \to \F$ is a map of sheaves preserving the parity and compatible with the
$\Oh_M$-module structure. A \define{bundle isomorphism} is an invertible bundle map
whose inverse is a bundle map. A bundle map is a special case of a graded bundle
map: a \define{graded bundle map of parity $p$} is a map of sheaves $\varphi \maps \E
\to \F$ that shifts the parity of sections by $p$:
\[ \varphi \maps \E_i(U) \to \F_{i+p}(U) , \]
for any open $U \subseteq M_b$, and that respects the $\Oh_M$-module structure up to
sign:
\[ \varphi(fs) = (-1)^{p|f|} f \varphi(s), \mbox{ for all } f \in \Oh_M(U), \, s \in
  \E(U) . \]
The sheaf $\hom(\E,\F)$ of all graded bundle maps is itself a vector bundle over
$M$.

We can define the usual operations on vector bundles in supergeometry \cite{Manin}:
\begin{itemize}
\item[ ] {\bf Duals.} Given a vector bundle $\E$ on $M$, its \define{dual} $\E^*$ is
  the sheaf of graded bundle maps from $\E$ into the structure sheaf: $\E^* =
  \hom(\E, \Oh_M)$.
\item[ ] {\bf Tensor product.} Given two vector bundles $\E$ and $\F$ on $M$, their
  \define{tensor product} $\E \otimes \F$ is a vector bundle on $M$, defined to be
  the sheaf with the stalks
  \[ (\E \otimes \F)_x = \E_x \otimes_{\Oh_{M,x}} \F_x, \]
  for $x \in M_b$ any point in the body.
\item[ ] {\bf Pullback.} Given a map $f \maps M \to N$ of supermanifolds and a vector
  bundle $\E$ on $N$, the \define{pullback} $f^* \E$ is a vector bundle on $M$,
  defined to be the sheaf with the stalks
  \[ (f^* \E)_x = \E_{f(x)} \otimes_{\Oh_{N,f(x)}} \Oh_{M,x} , \]
  for $x \in M_b$ any point in the body. In this tensor product, we treat the stalk
  $\Oh_{M,x}$ as an $\Oh_{N,f(x)}$-module using the homomorphism $f^*$.
\end{itemize}

As in classical differential geometry, given a vector bundle $\E$ on $M$, we can
construct a supermanifold $E$, the \define{total space} of $\E$, and a smooth map $p
\maps E \to M$, the \define{bundle projection} \cite{DM}. We will not work with these
objects directly, but will use them as a kind of shorthand to remind the reader of
the parallels with differential geometry. Thus we will sometimes speak of a vector
bundle $E$ or $E \to M$, or of a line bundle $L$ or $L \to M$, when we mean sheaves
$\E$ and $\L$ of $\Oh_M$-modules or $\Oh_\C$-modules, respectively. Similarly, we
will write $E^* \to M$ for the dual of $\E$, $E \otimes F \to M$ for the tensor
product of $\E$ and $\F$, and $f^* E \to M$ for the pullback of $\E$ on $N$ along $f
\maps M \to N$, but in all cases we mean the constructions with sheaves described
above.

\subsection{The de Rham complex}

The de Rham complex on $M$ will also play a central role, so we give a lightning
overview of its construction. The \define{tangent bundle} $\T M$ of $M$ is the sheaf
of graded derivations of $\Oh_M$, and we refer to global sections of $\T M$ as
\define{vector fields} on $M$. The \define{cotangent bundle} is the dual sheaf
$\Omega^1 (M) = \hom(\T M, \Oh_M)$, and we refer to global sections of $\Omega^1 (M)$
as \define{1-forms} on $M$. We write the canonical pairing between vector fields and
1-forms as:
\[ \langle -, - \rangle \maps \T M \otimes \Omega^1(M) \to \Oh_M . \]
The \define{de Rham complex} is the $\Z$-graded sheaf
$\Omega^\bullet(M)$ of super vector spaces whose $p$th grade is the $p$th exterior
power of $\Omega^1(M)$ as an $\Oh_M$-module:
\[ \Omega^p(M) := \Lambda^p_{\Oh_M}\left(\Omega^1 (M) \right) . \]
Note that each grade $\Omega^p(M)$ is itself $\Z_2$-graded, so in fact the full
complex $\Omega^\bullet(M)$ is $\Z \times \Z_2$-graded. Both degrees encode important
information about a form, and we will often specify them in words. For instance, if
we were to write ``let $\alpha$ be an odd 2-form'' or ``let $\alpha$ be a 2-form of
odd parity'', we would mean in both cases that $\alpha$ is an odd section of
$\Omega^2(M)$. We will also write $\Omega^p_i$ for the sheaf of $p$-forms of parity
$i$. Finally, we will also have occasion to think about the complexified de Rham
complex, $\Omega^p_\C = \Omega^p \otimes \C$, and its even part, $\Omega^p_{\C,0}$.

The full complex $\Omega^\bullet(M)$ is an algebra under the wedge product---in fact,
it is a $\Z \times \Z_2$-graded commutative algebra, so that $\alpha \wedge \beta =
(-1)^{pq + p'q'} \beta \wedge \alpha$, for $\alpha$ a $p$-form of parity $p'$ and
$\beta$ a $q$-form of parity $q'$. As classically, there is a differential $d \maps
\Oh_M \to \Omega^1(M)$ defined by
\[ \langle X, df \rangle = Xf , \]
for $X$ a local vector field and $f$ in the structure sheaf. This extends to a
derivation $d$ on all of $\Omega^\bullet(M)$ such that $d^2 = 0$, the
\define{exterior derivative} making $\Omega^\bullet(M)$ into a complex. The exterior
derivative is defined to have bidegree $(1,0)$. Thus for any $p$-form $\alpha$, its
exterior derivative $d\alpha$ is a $(p+1)$-form of the same parity as $\alpha$, and
we have $d(\alpha \wedge \beta) = d\alpha \wedge \beta + (-1)^p \alpha \wedge d
\beta$ for any other form $\beta$.

We now prove a key lemma concerning the relationship between the de Rham complex and
fiber powers of a submersion. This will play a central role when we work with
connections on bundle gerbes.

The following proof briefly invokes `$Z$-points' of the fiber power $Y^{[q]}$. For
any supermanifold $Z$, a \define{$Z$-point} $(y_1, y_2, \ldots, y_q) \maps Z \to
Y^{[q]}$ is simply a map into $Y^{[q]}$ from $Z$.  By the universal property of the
fiber product, any map into $Y^{[q]}$ decomposes into a $q$-tuple of maps $y_i \maps
Z \to Y$, such that $\pi y_1 = \pi y_2 = \cdots = \pi y_q$. We have anticipated this
in our notation by writing the $Z$-point as a $q$-tuple $(y_1, y_2, \ldots,
y_q)$. The utility of $Z$-points comes from the Yoneda lemma \cite{Riehl}; the Yoneda
Lemma tells us that in any category $C$, we can construct a morphism $f \maps x \to
y$ between two objects $x, y \in C$ by defining a function between the sets of
$z$-points $\Hom(z, x) \to \Hom(z,y)$ for any object $z \in C$, so long as these
functions fit together into a natural transformation from the functor $\Hom(-,x)$
to the functor $\Hom(-,y)$.

\begin{lem}
  \label{lem:murray}
  Let $\pi \maps Y \to M$ be a submersion, where $Y$ and $M$ are supermanifolds. Then
  for any fixed natural number $p$, the following complex of $p$-forms is exact:
  \[ 0 \longrightarrow \Omega^p(M) \stackrel{\pi^*}{\longrightarrow} \Omega^p(Y)
    \stackrel{\delta}{\longrightarrow} \Omega^p(Y^{[2]})
    \stackrel{\delta}{\longrightarrow} \cdots \stackrel{\delta}{\longrightarrow}
    \Omega^p(Y^{[q]}) \stackrel{\delta}{\longrightarrow} \cdots \]
  Here the differential $\delta \maps \Omega^p(Y^{[q]}) \to \Omega^p(Y^{[q+1]})$ is
  given by the sum $\displaystyle \delta = \sum_{i=1}^{q+1} (-1)^{i+1}\pi_i^*$, where
  $\pi_i \maps Y^{[q+1]} \to Y^{[q]}$ is the projection that omits that $i$th
  factor. 
\end{lem}

\begin{proof}
  The following proof adapts an argument due to Michael Murray \cite{Murray1} to the
  context of supermanifolds.
  
  First, it is not hard to check that $\delta \pi^*$ = 0 and more generally,
  $\delta^2 = 0$, so $\Omega^p(Y^{[\bullet]})$ is indeed a complex. To prove
  exactness, we will first assume that $\pi \maps Y \to M$ admits a global section,
  $s \maps M \to Y$. Afterwards, we will generalize to the case without a global
  section by using a partition of unity. In all of the following, we adopt the
  convention that $Y^{[0]} = M$ and $\delta = \pi^*$ at this stage of the complex. 

  Given a global section $s \maps M \to Y$, we get a map $s_q \maps Y^{[q]} \to
  Y^{[q+1]}$ defined by $s_q(y_1, \ldots, y_q) = (s(\pi(y_1)), y_1, \ldots, y_q)$,
  for any $Z$-point $(y_1, \ldots, y_q) \maps Z \to Y^{[q]}$. Let $\omega \in
  \Omega^p(Y^{[q]})$ be $\delta$-closed:
  \[ \delta \omega = 0 . \]
  We claim that $\omega = \delta s_{q-1}^* \omega$, proving exactness. Indeed, using
  $Z$-points it is easy to check that $\pi_1 s_q = 1_{Y^{[q]}}$ and $\pi_i s_q =
  s_{q-1} \pi_{i-1}$ for $i \geq 2$. Thus, by the definition of $\delta$ we have:
  \[ s_q^* \delta = \sum_{i = 1}^{q+1} (-1)^{i+1} s_q^* \pi_i^* = 1 - \sum_{i =
    2}^{q+1} (-1)^i \pi_{i-1}^* s_{q-1}^* = 1 - \delta s_{q-1}^* . \]
  So, because $\delta \omega = 0$, we must have $s_q^* \delta \omega = 0$, which
  implies $\omega - \delta s_{q-1}^* \omega = 0$ by the above calculation. Thus
  $\omega = \delta s_{q-1}^* \omega$, as claimed.

  Now let us consider the case where we only have local sections. Fix an open
  cover $\{ U_\alpha \}_{\alpha \in I}$ of $M$ and local sections $s_\alpha \maps
  U_\alpha \to Y$ of $\pi$. Over each open $U_\alpha$, define $Y_\alpha =
  \pi^{-1}(U_\alpha)$. Then the restriction $\pi \maps Y_\alpha \to U_\alpha$ is a
  surjective submersion with a global section, and we know by our work above that the
  following complex is exact:
  \[ 0 \longrightarrow \Omega^p(U_\alpha) \stackrel{\pi^*}{\longrightarrow}
    \Omega^p(Y_\alpha) \stackrel{\delta}{\longrightarrow} \Omega^p(Y_\alpha^{[2]})
    \stackrel{\delta}{\longrightarrow} \cdots \stackrel{\delta}{\longrightarrow}
    \Omega^p(Y_\alpha^{[q]}) \stackrel{\delta}{\longrightarrow} \cdots \]
  To extend this result to the full complex, choose a partition of unity
  $\{\varphi_\alpha \in \Oh(M) \}_{\alpha \in I}$ on $M$ subordinate to our cover $\{
  U_\alpha\}_{\alpha \in I}$, where $\Oh(M)$ denotes the global sections of the
  structure sheaf $\Oh_M$. Each fiber power $Y^{[q]}$ projects to $M$, and by pulling
  our partition of unity back along this projection, we get a partition of unity on
  $Y^{[q]}$ subordinate to the cover $\{Y^{[q]}_\alpha\}_{\alpha \in I}$. Let us
  abuse notation slightly and denote the pullback partition of unity by $\{
  \varphi_\alpha\}_{\alpha \in I}$ as well. Because they were constructed by
  pullback, multiplication by elements in our partition of unity commutes with
  $\delta$.

  Now let $\omega \in \Omega^p(Y^{[q]})$ be $\delta$-closed:
  \[ \delta \omega = 0 . \]
  Restricting $\omega$ to $Y^{[q]}_\alpha$, we get $\omega_\alpha \in
  \Omega^p(Y^{[q]}_\alpha)$, which is also $\delta$-closed, and thus $\omega_\alpha =
  \delta \rho_\alpha$ by exactness. Patching the $\rho_\alpha$ together using our
  partition of unity, it is immediate that \mbox{$\rho = \sum \varphi_\alpha
  \rho_\alpha$} satisfies $\omega = \delta \rho$.
\end{proof}

\subsection{Connections and curvature}

A \define{connection} $\nabla$ on a line bundle $\L$ on a supermanifold $M$ is a rule
that, for any vector field $X$ on $M$, gives a map of sheaves of super vector spaces,
$\nabla_X \maps \L \to \L$. This rule is $\Oh_M$-linear in $X$:
\[ \nabla_{fX} s = f \nabla_X s , \]
for any local section $f$ of the structure sheaf $\Oh_M$, and it satisfies the
Leibniz rule in a graded sense:
\[ \nabla_X (fs) = (Xf) s + (-1)^{|X||f|} f \nabla_X s , \]
for any local section $s$ of $\L$ and section $f$ of the complexified structure sheaf
$\Oh_\C$. The \define{curvature} $F$ of $\nabla$ is defined by the graded analogue of
the usual formula, namely:
\[ F(X,Y) s = \left( \nabla_X \nabla_Y - (-1)^{|X||Y|} \nabla_Y \nabla_X - \nabla_{[X,Y]} \right)
  s , \]
for any local section $s$ of $\L$. It is a complex-valued 2-form of even parity, $F
\in \ECDR^2(M)$.

\subsection{The body of a supermanifold}

In supermanifold theory, essentially all questions that concern the topology of $M$
reduce to questions on the body $M_b$, and we now introduce maps to pass between $M$
and its body. Despite $M_b$ being merely a topological manifold in our definition, it
in fact carries a smooth structure. This arises as follows: let $\J_M \subseteq
\Oh_M$ be the ideal generated by the odd subsheaf, $(\Oh_M)_1$. Then $(M_b,
\Oh_M/\J_M)$ is an ordinary smooth manifold: on each chart $(U, \varphi, \phi)$ of
our atlas, $\J_M|_U$ is generated by the odd coordinates $\{\theta_j\}$, so when we
quotient by this ideal we find that $(\Oh_M/\J_M)|_U$ is isomorphic to the sheaf of
smooth functions on $\varphi(U)$, giving $M_b$ a smooth structure. The inclusion map
$i \maps M_b \inclusion M$ is the map of supermanifolds that corresponds to the
quotient map of structure sheaves, $\Oh_M \to \Oh_M/\J_M$.

One way the topology of the body appears in our work is via open covers. An
\define{open cover} $\U = \{ U_\alpha \}$ of a supermanifold $M$ is an open cover of
the body $M_b$, where we equip each open $U_\alpha \subseteq M_b$ with the
supermanifold structure $(U_\alpha, \Oh_M|_{U_\alpha})$ given by restricting the
structure sheaf. We say that $\U$ is a \define{good cover} of $M$ if it is a good
cover of $M_b$, in the sense that each nonempty finite intersection:
\[ U_{\alpha_1 \alpha_2 \cdots \alpha_k} = U_{\alpha_1} \cap U_{\alpha_2} \cap \cdots
  \cap U_{\alpha_k} \]
is diffeomorphic to $\R^m$, for $m$ the dimension of $M_b$.

Henceforth, we shall always regard the body $M_b$ as a smooth submanifold of
$M$. Note that the inclusion of the body, $i \maps M_b \inclusion M$, is canonical
for any supermanifold, but a map from $M$ to its body is not. We say a smooth map $p
\maps M \to M_b$ is a \define{body projection} if $p i = 1_{M_b}$.  It turns out that
such a map always exists \cite{Batchelor, Gawedzki}, but it is a choice of extra data.

\section{Bundle gerbes}
\label{sec:gerbes}

The notion of bundle gerbe is due to Michael Murray \cite{Murray}. It is a geometric
object defined over a base manifold whose characteristic class sits in the third
cohomology of the base. It is analogous to a complex line bundle, and its
characteristic class is an analogue of the Chern class in degree 2. A line bundle can
be equipped with a connection, and when this is done the curvature of the connection
is a closed 2-form and a de Rham representative of the Chern class. Analogously, a
bundle gerbe can also be equipped with a connection, and when this is done the
curvature is a closed 3-form and a de Rham representative the characteristic class.

Bundle gerbes are just one of many constructions that geometrize 3rd degree
cohomology. Shortly before Murray's work, Brylinski developed the differential
geometry of objects called `gerbes', by which he meant a sheaf of groupoids
\cite{Brylinski}. In this work, Brylinski was using the original notion of gerbe due
to Giraud \cite{Giraud}, conceived as a formulation of nonabelian
cohomology. Contemporaneous with Giraud, Deligne developed the tool now called
Deligne cohomology \cite[Sec. 2.2]{Deligne}, the cohomology theory that we shall see
classifies bundle gerbes with connection. A few years after Murray, Chatterjee and
Hitchin worked with a restricted class of bundle gerbes they called `gerbs'
\cite{Chatterjee, Hitchin}.

In Sections \ref{sec:noconn} and \ref{sec:conn}, we invoke a dictionary between
classical differential geometry and supergeometry to translate the definitions and
theorems of Murray \cite{Murray} and Murray--Stevenson \cite{MurrayStevenson} to
supergeometric language. To my knowledge, this dictionary has not been made explicit
before, but it is apparent in the work of Kostant \cite{Kostant} on line bundles in
supergeometry. To wit, a line bundle with connection on an ordinary manifold $X$ has
transition functions given by local sections of the sheaf $\underline{\C}^*_X$ of
smooth functions valued in the nonzero complex numbers, and its connection is locally
a complex-valued 1-form. On the other hand, a line bundle with connection on a
supermanifold $M$ has transition elements given by local sections of $\Oh^*_\C$, the
subsheaf of even, invertible elements of the complexified structure sheaf, $\Oh_\C =
\C \otimes \Oh_M$, and the connection is locally an even, complex-valued
1-form. Bundle gerbes involve higher degree forms, but the dictionary is the same:
replace the sheaf of smooth functions $\underline{\C}^*_X$ with $\Oh_\C^*$, and the
sheaf of complexified de Rham forms $\Omega^\bullet_\C$ with its even part
$\ECDR^\bullet$.

\subsection{Without connection}
\label{sec:noconn}

Recall from Section \ref{sec:superman} that for any submersion $\pi \maps Y \to M$,
the fiber powers $Y^{[p]} = Y \times_M Y \times_M \cdots \times_M Y$ ($p$ times)
exist as supermanifolds. Between the fiber powers, we have the maps $\pi_i \maps
Y^{[p]} \to Y^{[p-1]}$ which omit the $i$th factor, and more generally we have the
maps $\varpi_{i_1, i_2, \ldots, i_q} \maps Y^{[p]} \to Y^{[q]}$ that project onto the
$i_1$th, $i_2$th, \ldots, $i_q$th factors, in that order, for any $1 \leq i_1 < i_2 <
\cdots < i_q \leq p$.

\begin{defn}
  A \define{bundle gerbe} on the supermanifold $M$ consists of the following data:
  \begin{itemize}
  \item a submersion, $\pi \maps Y \to M$.
  \item a complex line bundle over the fiber square of $Y$:
    \[
      \begin{tikzcd}
        L \ar{d} & \\
        Y^{[2]} \ar[r, shift left, "\pi_1"] \ar[r, shift right, "\pi_2"'] & Y \ar[d, "\pi"] \\
        & M . \\
      \end{tikzcd}
    \]
  \item a line bundle isomorphism, the bundle gerbe \define{multiplication}, $\mu
    \maps \pi_3^* L \otimes \pi_1^* L \to \pi_2^* L$ over $Y^{[3]}$. Note that we can
    equivalently write this as $\mu \maps \varpi_{12}^* L \otimes \varpi_{23}^* L \to
    \varpi_{13}^* L$.
  \item The bundle gerbe multiplication satisfies an associativity condition over
    $Y^{[4]}$: using $\mu$, we can construct two line bundle isomorphisms over
    $Y^{[4]}$ from the line bundle $\varpi^*_{12} L \otimes \varpi^*_{23} L \otimes
    \varpi^*_{34} L$ to the line bundle $\varpi^*_{14} L$. These two maps are required to
    be equal. Explicitly, the following square commutes:
    \[
      \begin{tikzcd}
        \varpi^*_{12} L \otimes \varpi^*_{23} L \otimes \varpi^*_{34} L \ar[r, "\varpi_{123}^* \mu \otimes 1"] \ar[d, "1 \otimes \mu"] & \varpi^*_{13} L \otimes \varpi^*_{34} L \ar[d, "\varpi_{134}^* \mu"] \\
        \varpi^*_{12} L \otimes \varpi^*_{24} L \ar[r, "\varpi_{124}^* \mu"] & \varpi^*_{14} L .
      \end{tikzcd} 
    \]
  \end{itemize}
\end{defn}

\noindent
We denote a bundle gerbe with surjective submersion $Y$, line bundle $L$ and
multiplication $\mu$ by the triple $(Y, L, \mu)$, or more succinctly by a single
calligraphic letter such as $\G$. The following is the most automatic example of a
bundle gerbe, called the `trivial bundle gerbe':
\begin{ex}[The trivial bundle gerbe] \label{ex:triv}
  Given a submersion $\pi \maps Y \to M$, choose a line bundle $L$ over $Y$. Then we
  can construct a line bundle $\delta L$ on $Y^{[2]}$ as follows:
  \[ \delta L = \pi_1^* L \otimes \pi_2^* L^* , \]
  where we recall that $\pi_i \maps Y^{[2]} \to Y$ denotes the projection that omits
  the $i$th factor. There is a isomorphism of line bundles $\mu_{\rm can} \maps
  \pi_3^* \delta L \otimes \pi_1^* \delta L \to \pi_2^* \delta L$ over $Y^{[3]}$,
  induced from the canonical pairing between $L$ and its dual $L^*$, and it is a
  quick calculation to check that $\mu_{\rm can}$ is associative. Thus $(Y, \delta L,
  \mu_{\rm can})$ is a bundle gerbe over $M$. We call this \define{the trivial bundle
  gerbe} over $M$, and denote it by $\I$.  We shall also denote this gerbe by $\delta
  L$ when we wish to make note of the line bundle, though we shall see shortly that
  all trivial bundle gerbes over $M$ are equivalent to each other, independent of the
  choice of line bundle or submersion. Hence the notation $\I$ will be preferred.
\end{ex}

\noindent
All the usual operations on line bundles generalize to bundle gerbes:
\begin{itemize}
\item[ ] {\bf Duals.} Given a bundle gerbe $\G = (Y, L, \mu)$ over $M$, its
  \define{dual} $\G^*$ is the bundle gerbe $(Y, L^*, (\mu^{-1})^*)$ over $M$.
\item[ ] {\bf Tensor product.} Given two bundle gerbes $\G = (Y, L, \mu)$ and $\H =
  (Z, Q, \nu)$, the \define{tensor product} $\G \otimes \H$ is the bundle gerbe $(Y
  \times_M Z, L \otimes Q, \mu \otimes \nu)$. Here $L \otimes Q$ denotes the result
  of tensoring the pullbacks of $L$ and $Q$ to the fiber square $(Y \times_M
  Z)^{[2]}$, but we treat the pullbacks as implicit.
\item[ ] {\bf Pullback.} Given a bundle gerbe $\G = (Y, L, \mu)$ over the
  supermanifold $N$, and a map of supermanifolds $f \maps M \to N$, the
  \define{pullback} $f^* \G$ is the bundle gerbe $(f^* Y, f^* L, f^* \mu)$, where we
  have abused notation to write $f^* L$ for the pullback of the line bundle $L$ to
  $(f^* Y)^{[2]}$, and similarly for the multiplication.
\end{itemize}

\noindent
The collection of all bundle gerbes over $M$ forms a bicategory rather than a mere
category. In a full treatment, we would now describe the 1-morphisms and 2-morphisms
in this bicategory. This has been done by Stevenson \cite{Stevenson} and Waldorf
\cite{Waldorf} for bundle gerbes over smooth manifolds, however, and we do not
anticipate the 1- and 2-morphisms over supermanifolds to differ
significantly. Nonetheless, we will need one piece of information about the
bicategory of bundle gerbes: we need to know when two bundle gerbes over $M$ are
equivalent. We turn to this now.

A \define{trivialization} of a bundle gerbe $\G = (Y, L, \mu)$ is a choice of line
bundle $T \to Y$ and an isomorphism $\tau \maps L \to \delta T$ of line bundles over
$Y^{[2]}$ compatible with the bundle gerbe multiplication on $\G$ and $\delta T$. We
say that a bundle gerbe is \define{trivializable} if a trivialization exists.
\begin{defn}
  Two bundle gerbes $\G$ and $\H$ over $M$ are \define{equivalent} if $\G \otimes
  \H^*$ is trivializable. 
\end{defn}
\noindent
We write $\G \simeq \H$ when $\G$ and $\H$ are equivalent. Equivalence of bundle
gerbes is also called \define{stable isomorphism} in the literature. With this
definition in hand, it is easy to see that the trivial bundle gerbe is unique up to
equivalence.

\begin{prop}\label{prop:triv}
  Let $Y \to M$ and $Z \to M$ be submersions, and let $L$ and $Q$ be line bundles
  over $Y$ and $Z$, respectively. Then the trivial bundle gerbes $(Y, \delta L,
  \mu_{\rm can})$ and $(Z, \delta Q, \mu_{\rm can})$ are equivalent.
\end{prop}

\begin{proof}
  Write $\delta L$ and $\delta Q$, respectively, for the above bundle gerbes. We
  merely need to show that $\delta L \otimes (\delta Q)^*$ is trivializable. The
  submersion for this tensor product of gerbes is $X = Y \times_M Z$, and in fact
  we have an isomorphism of line bundles
  \[ \delta L \otimes (\delta Q)^* \iso \delta( L \otimes Q^* ) \]
  over the fiber square, $X^{[2]}$. This is the desired trivialization.
\end{proof}

\begin{cor}
  \label{cor:triv}
  The trivial bundle gerbe $\I$ is equivalent to the following bundle gerbe:
  \[  (M, \Oh_\C, \mu_{\rm can}) \]
  where the submersion $M \to M$ is the identity, the line bundle over $M^{[2]} = M$
  is the trivial line bundle $\Oh_\C$, and the multiplication $\mu_{\rm can}$ is the
  usual multiplication on the complexified structure sheaf, $\Oh_\C$.
\end{cor}

Let $\mathbb{G}(M)$ denote the collection of all equivalence classes of bundle gerbes
on $M$. The tensor product operation equips $\mathbb{G}(M)$ with an abelian group
structure, with duals as inverses and the trivial bundle gerbe $\I$ as the
identity. In fact, we shall see that $\mathbb{G}(M)$ depends only on the topology of
the body:
\[ \mathbb{G}(M) \iso H^3(M_b, \Z) . \]
For a given bundle gerbe $\G$, the corresponding class in $H^3(M_b, \Z)$ is called
the `Dixmier--Douady class of $\G$', and denoted $\DD(\G)$. It is the analogue
for bundle gerbes of the Chern class of a line bundle. As in the line bundle case,
the Dixmier--Douady class is constructed using \Cech\ cohomology.

We perform the construction of $\DD(\G)$ for $\G = (Y, L, \mu)$ as follows. Choose a
good cover $\{U_\alpha\}$ of $M$ such that the submersion $\pi \maps Y \to M$
admits local sections $s_\alpha \maps U_\alpha \to Y$. On each double
intersection $U_{\alpha\beta} = U_\alpha \cap U_\beta$, the universal property of the
fiber product yields a map:
\[ (s_\alpha, s_\beta) \maps U_{\alpha \beta} \to Y^{[2]} . \]
We use the map $(s_\alpha, s_\beta)$ to pull back our line bundle $L$ to
$U_{\alpha \beta}$, yielding a line bundle $L_{\alpha \beta} = (s_\alpha, s_\beta)^*
L$. Because $U_{\alpha \beta}$ is contractible, $L_{\alpha \beta}$ is trivializable
\cite{Kostant}, so there is a section $\sigma_{\alpha \beta} \in \L_{\alpha
\beta}(U_{\alpha \beta})$ such that $\L_{\alpha \beta} = \Oh_\C \cdot
\sigma_{\alpha\beta}|_{U_{\alpha\beta}}$. On the triple intersection $U_{\alpha \beta
\gamma} = U_\alpha \cap U_\beta \cap U_\gamma$, the bundle gerbe multiplication
induces an isomorphism $\mu \maps L_{\alpha \beta} \otimes L_{\beta \gamma} \to
L_{\alpha \gamma}$. We thus have two sections of $L_{\alpha \gamma}$ over a triple
intersection, namely $\sigma_{\alpha \gamma}$ and $\mu(\sigma_{\alpha \beta} \otimes
\sigma_{\beta \gamma})$, both of which are trivializing. Thus we must have:
\[ \mu(\sigma_{\alpha \beta} \otimes \sigma_{\beta \gamma}) = g_{\alpha \beta \gamma}
  \sigma_{\alpha \gamma} , \]
for some even, invertible element of the complexified structure sheaf, $g_{\alpha
\beta \gamma} \in \Oh_\C(U_{\alpha \beta \gamma})$. The associativity of $\mu$
implies that $g_{\alpha \beta \gamma}$ is a \Cech\ 2-cocycle. If we write
$\Oh_\C^*$ for the subsheaf of even, invertible elements in the complexified
structure sheaf $\Oh_\C$, then we have constructed a class in the 2nd \Cech\
cohomology of this sheaf:
\[ [g_{\alpha \beta \gamma}] \in \check{H}^2(M_b, \Oh_\C^*) . \]
By standard arguments in \Cech\ cohomology, we can check that the class $[g_{\alpha
\beta \gamma}]$ is independent of the choices made.

Finally, we can construct an isomorphism $\check{H}^2(M_b, \Oh^*_\C) \iso
H^3(M_b, \Z)$, so that $[g_{\alpha \beta \gamma}]$ defines a class in
$H^3(M_b,\Z)$. This is the \define{Dixmier--Douady class of $\G$}, denoted
$\DD(\G)$. Thanks to the isomorphism with $\check{H}^2(M_b, \Oh^*_\C)$, we will
also think of $\DD(\G)$ as living in this group.

We construct the isomorphism $\check{H}^2(M_b, \Oh^*_\C) \iso H^3(M_b, \Z)$ in the
same way as one constructs the classical isomorphism $\check{H}^p(X,
\underline{\C}^*_X) \iso H^{p+1}(X, \Z)$, for $X$ a smooth manifold and
$\underline{\C}^*_X$ the sheaf of smooth functions valued in the invertible complex
numbers, $\C^*$. Specifically, we use a short exact sequence of sheaves on $M_b$:
\[ 0 \longrightarrow \Z_{M_b} \stackrel{2\pi i}{\longrightarrow} (\Oh_\C)_0
  \stackrel{\exp}{\longrightarrow} \Oh^*_\C \longrightarrow 0 . \]
Indeed, Kostant proved that if we define $\exp(f) = \sum_{n = 0}^\infty
\frac{f^n}{n!}$, then this series converges in a suitable topology and defines a
homomorphism $\exp \maps (\Oh_\C)_0(U) \to \Oh_\C^*(U)$ for any open set $U$, which
is surjective when $U$ is simply connected. Moreover, he showed this descends to a
sheaf homomorphism $\exp \maps (\Oh_\C)_0 \to \Oh^*_\C$, with kernel given by the
constant sheaf on the integers \cite{Kostant}.

With this short exact sequence of sheaves in hand, we get a long exact sequence upon
passing to sheaf cohomology. Because $(\Oh_\C)_0$ is a soft sheaf, the connecting
homomorphism provides the desired isomorphism $\check{H}^2(M_b, \Oh^*_\C) \iso
H^3(M_b, \Z)$.

To close this circle of ideas, we now prove that the Dixmier--Douady class just
constructed is a complete invariant of a bundle gerbe up to equivalence. This is the
content of the following theorem. Its proof follows the proof of the analogous result
over smooth manifolds by Murray \cite{Murray} and Murray--Stevenson
\cite{MurrayStevenson}, adapted to the language of structure sheaves.

\begin{thm} \label{thm:cech} Let $M$ be a supermanifold, $\mathbb{G}(M)$ the abelian
  group of equivalence classes of bundle gerbes over $M$, and let $\DD \maps
  \mathbb{G}(M) \to \check{H}^2(M, \Oh^*_\C)$ be the map sending the equivalence
  class of a bundle gerbe on $M$ to its Dixmier--Douady class, regarded as a \Cech\
  cohomology class. Then:
  \begin{enumerate}
  \item $\DD \maps \mathbb{G}(M) \to \check{H}^2(M, \Oh^*_\C)$ is an isomorphism of
    abelian groups, i.e., 
    \begin{itemize}
    \item $\DD$ is onto;
    \item $\DD(\G) = 0$ if and only if $\G$ is trivial;
    \item $\DD( \G \otimes \H ) = \DD(\G) + \DD(\H)$;
    \item $\DD(\G^*) = -\DD(\G)$;
    \end{itemize}
    where $\G$ and $\H$ are bundle gerbes on $M$. 
  \item If $f \maps M \to N$ is a map of supermanifolds, then $\DD$ is natural with respect
    to pullback along $f$:
    \[ \DD(f^* \G) = f^* \DD(\G), \]
    for any bundle gerbe $\G$ on $N$.
  \end{enumerate}
\end{thm}

\begin{proof}
  Let $\G = (Y, L, \mu)$ be a bundle gerbe on $M$, and let $\DD(\G) = [g_{\alpha
  \beta \gamma}]$ be the \Cech\ cohomology class constructed as above using a good
  cover $\{U_\alpha\}$, sections $s_\alpha \maps U_\alpha \to Y$, and trivializing
  sections $\sigma_{\alpha \beta} \in \L_{\alpha\beta}(U_{\alpha\beta})$. As we noted
  above, the class $\DD(\G)$ is independent of the choices made by standard arguments
  in \Cech\ cohomology.

  The fact that $\DD$ respects pullbacks is also standard: construct $\{ g_{\alpha
  \beta \gamma} \}$ using a good cover of $\{U_\alpha\}$ of $N$, and choose a good
  cover $\{V_\delta\}$ on $M$ that is a refinement of the cover
  $\{f^{-1}(U_\alpha)\}$. The pullback of $\{ g_{\alpha \beta \gamma} \}$ to this
  refinement represents the class of $\DD(f^* \G)$, as desired.

  Next, we wish to show that $\DD$ is an isomorphism of abelian groups. First of all,
  the fact that $\DD$ is a group homomorphism:
  \[ \DD(\G \otimes \H) = \DD(\G) + \DD(\H) , \]
  is a quick calculation with \Cech\ representatives, for any two bundle gerbes $\G$
  and $\H$, taking care to write the group operation additively: $[g_{\alpha \beta
  \gamma}] + [h_{\alpha \beta \gamma}] = [ g_{\alpha \beta \gamma} h_{\alpha \beta
  \gamma}]$.  From this it follows that $\DD(\I) = 0$ for the trivial bundle gerbe
  $\I$ and that $\DD(\G^*) = -\DD(\G)$.

  So, it remains to check that $\DD$ is a bijection. To show it is onto, fix a class
  $\underline{g} \in \check{H}^2(M_b, \Oh^*_\C)$, and choose a good cover
  $\{U_\alpha\}$ for which $\underline{g}$ is represented by elements $g_{\alpha
  \beta \gamma} \in \Oh_\C(U_{\alpha \beta \gamma})$ satisfying the \Cech\ 2-cocycle
  condition. Let our submersion be $Y = \coprod (U_\alpha, \Oh|_{U_\alpha})$ with the
  obvious projection to $M$. Then the fiber square is $Y^{[2]} = \coprod
  (U_{\alpha \beta}, \Oh|_{U_{\alpha \beta}})$. For our bundle gerbe $\G$, we take
  the trivial line bundle on $Y^{[2]}$ and use $\underline{g}$ to define the bundle
  gerbe multiplication over $Y^{[3]}$, which is associative because $g_{\alpha \beta
  \gamma}$ is a 2-cocycle. We have that $\DD(\G) = \underline{g}$ by construction.

  Finally, to show that $\DD$ is one-to-one, we check that it has trivial
  kernel. That is, $\DD(\G) = 0$ implies $\G$ has a trivialization. Indeed, $\DD(\G)
  = 0$ means our \Cech\ 2-cocycle $\{g_{\alpha \beta \gamma}\}$ is a coboundary:
  \[ g_{\alpha \beta \gamma} = f_{\alpha \beta} f_{\beta \gamma} f_{\gamma \alpha}
    , \]
  for some even, invertible sections $f_{\alpha \beta} \in \Oh_\C(U_{\alpha
  \beta})$. In fact, without loss of generality, we can assume $g_{\alpha \beta
  \gamma} = 1$. Simply replace the trivializing sections $\sigma_{\alpha \beta} \in
  \L_{\alpha \beta}(U_{\alpha \beta})$ with $\frac{1}{f_{\alpha \beta}}
  \sigma_{\alpha \beta}$. For these trivializing sections, we indeed have $g_{\alpha
  \beta \gamma} = 1$.
  
  We can now use these data to construct a line bundle $T \to Y$ by patching together
  line bundles $T_\alpha$ on $Y_\alpha = \pi^{-1}(U_\alpha)$. Each $T_\alpha$ is
  defined as follows: note that there are two maps from $Y_\alpha$ to $Y$: the
  inclusion $i \maps Y_\alpha \to Y$ and the composite of the projection followed by
  the section, $s_\alpha \pi \maps Y_\alpha \to Y$. By the universal property of the
  fiber product, this gives us a map $(s_\alpha \pi, i) \maps Y_\alpha \to
  Y^{[2]}$. So, we can pull $L$ back along this map, and define $T_\alpha = (s_\alpha
  \pi, i)^* L$. These glue together to give a line bundle $T$ on all of $Y$. To see
  this, note that using the sections $\sigma_{\alpha \beta}$ on each intersection
  $Y_{\alpha \beta}$, we can define a line bundle isomorphism:
  \[ \tilde{\sigma}_{\alpha \beta} \maps T_{\alpha} \to T_{\beta} . \]
  How? Because $\sigma_{\alpha \beta} \in \L_{\alpha \beta}$, and using the bundle
  gerbe multiplication, this is isomorphic to $(s_\alpha \pi, i)^* \L \otimes
  (s_\beta \pi, i)^* \L^* = \T_\alpha \otimes \T^*_\beta$. Finally, because we have
  normalized the sections $\sigma_{\alpha \beta}$ so that $g_{\alpha \beta \gamma} =
  1$, these isomorphisms satisfy:
  \[ \tilde{\sigma}_{\alpha \beta} \tilde{\sigma}_{\beta \gamma}
    \tilde{\sigma}_{\gamma \alpha} = 1 \]
  and we conclude that the $T_\alpha$ do indeed glue to give a line bundle $T \to
  Y$. By construction, $L \iso \delta T$ and this is compatible with the bundle gerbe
  multiplication, so $T$ defines a trivialization of $\G$.
\end{proof}

\subsection{With connection}
\label{sec:conn}

\begin{defn}
  A \define{connection} on a bundle gerbe $\G = (Y, L, \mu)$ consists of the
  following data:
  \begin{itemize}
  \item a connection $\nabla$ on the line bundle $L$, called the \define{line bundle connection};
  \item an even, complex-valued 2-form $B \in \Omega^2_{\C,0}(Y)$, called the
    \define{curving}; here $\Omega_\C^p$ denotes the complexification $\Omega^p
    \otimes \C$, and $\Omega^p_{\C,0}$ denotes its even part;
  \end{itemize}
  such that:
  \begin{itemize}
  \item the gerbe multiplication $\mu \maps \pi_3^* L \otimes \pi_1^* L \to \pi_2^*
    L$ is an isomorphism of line bundles with connection over $Y^{[3]}$, for the
    connection induced by $\nabla$;
  \item the line bundle connection and curving satisfy the \define{descent equation}
    on $Y^{[2]}$:
    \[ F_\nabla = \pi_1^* B - \pi_2^* B , \]
    where $F_\nabla$ is the curvature 2-form of $\nabla$.
  \end{itemize}
\end{defn}
\noindent
We write a connection as the pair $(\nabla, B)$ consisting of the line bundle
connection and the curving. We denote a bundle gerbe $\G$ with connection $(\nabla,
B)$ by the triple $(\G, \nabla, B)$, or simply by $\G$. While we call the pair
$(\nabla, B)$ a connection, some authors use the word ``connection'' to refer to the
line bundle connection $\nabla$ alone \cite{Murray, MurrayStevenson}. In the theory
of 2-bundles \cite{Bartels}, which generalizes the theory of bundle gerbes
\cite{NikolausWaldorf}, the pair $(\nabla, B)$ is called a \define{2-connection}
\cite{BaezHuerta, BaezSchreiber}.

The first example of a connection is on the trivial bundle gerbe, and we can produce
such a connection starting from any even 2-form on $M$. This is analogous to the fact
that an even 1-form defines a connection on the trivial line bundle.

\begin{ex}[Connections on the trivial bundle gerbe]
  \label{ex:trivconn}
  By Corollary \ref{cor:triv}, the trivial bundle gerbe $\I$ has the following simple
  form, up to equivalence:
  \[ \I = (M, \Oh_\C, \mu_{\rm can} ) \]
  where the submersion $M \to M$ is the identity, the line bundle over $M^{[2]} = M$
  is the trivial line bundle $\Oh_\C$, and the multiplication $\mu_{\rm can}$ is the
  usual multiplication on the complexified structure sheaf, $\Oh_\C$.

  Now fix an even, complex-valued 2-form $b \in \Omega^2_{\C,0}(M)$. Using $b$, we
  equip the trivial bundle gerbe with a connection as follows: take the line bundle
  connection to be the canonical flat connection on $\nabla^0$ on $\Oh_\C$, and take
  the curving to be the 2-form $b$. We denote the trivial bundle gerbe with this
  connection by $\I_b$, and call it \define{the trivial bundle gerbe with connection
  2-form $b$}. For $b = 0$, we continue to write $\I$ rather than $\I_0$.

  A trivial line bundle with connection, considered up to isomorphism, is the same as
  a 1-form. For bundle gerbes, however, the analogous fact is not true: $\I_b \simeq
  \I_{b'}$ does not imply that $b = b'$, but rather that $b - b'$ is a closed,
  integral 2-form, as we explain later.
\end{ex}

\noindent
All the usual operations on bundle gerbes can be easily modified to incorporate
connections:
\begin{itemize}
\item[ ] {\bf Duals.} Given a bundle gerbe $\G$ over $M$ with connection $(\nabla,
  B)$, its dual $\G^*$ has the connection $(\nabla^*, -B)$.
\item[ ] {\bf Tensor product.} Given two bundle gerbes $\G = (Y, L, \mu)$ and $\H =
  (Z, Q, \nu)$ with connections $(\nabla, B)$ and $(\nabla', B')$, respectively,
  recall their tensor product is
  \[ \G \otimes \H = (Y \times_M Z, L \otimes Q, \mu \otimes \nu) . \]
  This has the connection $(\nabla + \nabla', B + B')$. Here $B + B'$ denotes the
  result of pulling $B$ and $B'$ back to $Y \times_M Z$ and adding them, but we treat
  these pullbacks as implicit.
\item[ ] {\bf Pullback.} Given a bundle gerbe $\G$ with connection $(\nabla, B)$ over
  the supermanifold $N$, and a map of supermanifolds $f \maps M \to N$, the pullback
  $f^* \G$ has the connection $(f^* \nabla, f^* B)$.
\end{itemize}

\noindent
Much like bundle gerbes without connection, the collection of bundle gerbes with
connection on $M$ form a bicategory \cite{Stevenson, Waldorf}. For our purposes, it
is enough to know when two bundle gerbes with connection are equivalent in this
bicategory. For bundle gerbes without connection, $\G$ and $\H$ are equivalent when
$\G \otimes \H^*$ is trivializable, and the definition with connection is the
same. We need only specify how connections behave under trivialization.

A \define{trivialization} of a bundle gerbe $\G = (Y, L, \mu)$ with connection
$(\nabla, B)$ is a choice of the following data:
\begin{itemize}
\item a line bundle $T \to Y$ with connection;
\item an isomorphism $\tau \maps L \to \delta T$ of line bundles with connection
  over $Y^{[2]}$.
\end{itemize}
These data must satisfy:
\begin{itemize}
\item The isomorphism $\tau$ respects the bundle gerbe multiplication on $\G$ and $\delta T$;
\item The curvature $F_T$ of the connection on $T$ equals the curving: $F_{T} = B$.
\end{itemize}
We say that a bundle gerbe with connection is \define{trivializable} if a
trivialization exists.
\begin{defn}
  Two bundle gerbes $\G$ and $\H$ with connection over $M$ are \define{equivalent} if
  $\G \otimes \H^*$ is trivializable.
\end{defn}
\noindent
As before, we write $\G \simeq \H$ when $\G$ and $\H$ are equivalent. 

Let $\mathbb{G}^{\nabla}(M)$ denote the collection of all equivalence classes of
bundle gerbes on $M$. The tensor product operation equips $\mathbb{G}^{\nabla}(M)$
with an abelian group structure, with duals as inverses and $\I$ as the identity. In
contrast to the case without connection, $\mathbb{G}^{\nabla}(M)$ depends on the
supermanifold structure of $M$, not just on the topology of the body. Like their line
bundle cousins, bundle gerbes with connection are classified by `Deligne cohomology':
\[ \mathbb{G}^{\nabla}(M) \iso H^2(M,\Oh_\C^* \stackrel{d\log}{\longrightarrow}
  \Omega^1_{\C,0} \stackrel{d}{\longrightarrow} \Omega^2_{\C,0}) .
\]
Given a bundle gerbe $\G$, we denote the corresponding cohomology class by
$\Del(\G)$, and call it the `Deligne class of $\G$.' To describe its construction, we
must first define Deligne cohomology for supermanifolds. 

Recall that for $X$ a smooth manifold, the \define{$k$th smooth Deligne complex} is
the following complex of sheaves on $X$ (with $\underline{\C}^*_X$ in degree 0):
\[ \underline{\C}^*_X \stackrel{d\log}{\longrightarrow}
  \Omega_\C^1 \stackrel{d}{\longrightarrow} \Omega^2_\C \stackrel{d}{\longrightarrow} \cdots \stackrel{d}{\longrightarrow} \Omega^k_\C . \]
We then define \define{smooth Deligne cohomology} on $X$ as the hypercohomology of
this complex.

Concretely, we can compute Deligne cohomology using \Cech\ hypercohomology: fix a
good cover $\{U_\alpha\}$ of $X$, and take the \Cech\ complex for each sheaf in the
Deligne complex. In this way, we obtain a double complex of abelian groups, called
the \define{\Cech--Deligne complex}. Totalizing and taking cohomology yields Deligne
cohomology.

It is not difficult to generalize this definition to supermanifolds, with one minor
complication: we will need only the even part of the de Rham complex. So, on a
supermanifold $M$, define the \define{$k$th smooth Deligne complex $\DelCpx(k)$} to be the
following complex of sheaves (with $\Oh_\C^*$ in degree 0):
\[ \Oh^*_\C \stackrel{d\log}{\longrightarrow} \Omega^1_{\C,0}
  \stackrel{d}{\longrightarrow} \Omega^2_{\C,0} \stackrel{d}{\longrightarrow} \cdots
  \stackrel{d}{\longrightarrow} \Omega^k_{\C,0} . \]
As usual, $\Oh^*_\C$ is the subsheaf of even, invertible elements of the complexified
structure sheaf $\Oh_\C$, and $\Omega^p_{\C,0}$ is the sheaf of even, complexified
$p$-forms. We define the operation $d\log$ to mean
\[ d\log(f) = \frac{df}{f} , \]
for any $f \in \Oh^*_\C(U)$. Because sections of $\Oh^*_\C$ are even and invertible,
this definition is unambiguous.

The \define{smooth Deligne cohomology} of $M$ is the hypercohomology of the $k$th
Deligne complex for a given $k$. We will focus on the second cohomology group of the
second smooth Deligne complex, $H^2(M, \DelCpx(2))$, because this is the group
that classifies bundle gerbes with connection. For brevity, we call this group the
\define{Deligne cohomology} of $M$.

We now construct a class in the Deligne cohomology of $M$ from a bundle gerbe $\G =
(Y, L, \mu)$ with connection $(\nabla, B)$. We use \Cech\ hypercohomology to do
so. The choices involved are the same as in the previous section, but we recall them
here: choose a good open cover $\mathcal{U} = \{U_\alpha\}$ of $M$ such that the
submersion $\pi \maps Y \to M$ admits sections $s_\alpha \maps U_\alpha \to Y$.  On
each double intersection $U_{\alpha\beta} = U_\alpha \cap U_\beta$, the universal
property of the fiber product yields a map:
\[ (s_\alpha, s_\beta) \maps U_{\alpha \beta} \to Y^{[2]} . \]
As in the last section, in the map $s_\alpha \maps U_\alpha \to Y$ and the induced
map $(s_\alpha, s_\beta) \maps U_{\alpha \beta} \to Y^{[2]}$, we are treating open
subsets of the body $M_b$ as supermanifolds. We do this by restricting the structure
sheaf $\Oh_M$ to the given open set.

Now use the map $(s_\alpha, s_\beta)$ to pull back our line bundle $L$ to $U_{\alpha
\beta}$, yielding a line bundle $L_{\alpha \beta} = (s_\alpha, s_\beta)^* L$. Because
$U_{\alpha \beta}$ is contractible, $L_{\alpha \beta}$ is trivializable
\cite{Kostant}, so there is a section $\sigma_{\alpha \beta} \in \L_{\alpha
\beta}(U_{\alpha \beta})$ such that $\L_{\alpha \beta} = \Oh_\C \cdot
\sigma_{\alpha\beta}|_{U_{\alpha\beta}}$.

Having made these choices, we obtain our degree 2 class in Deligne cohomology. In
detail, we take the \Cech--Deligne double complex
\[ C^\bullet(\mathcal{U},\DelCpx(2)) \]
associated to the cover $\mathcal{U}$, and pass to the total complex. Degree 2 of the
total complex is the following direct sum:
\[ \prod_{\alpha,\beta,\gamma} \Oh^*_\C(U_{\alpha\beta\gamma}) \oplus \,
  \prod_{\alpha,\beta} \Omega^1_{\C,0}(U_{\alpha\beta}) \oplus \, \prod_{\alpha}
  \Omega^2_{\C,0}(U_\alpha) . \]
Inside this direct sum we find our representative $(g_{\alpha\beta\gamma},
A_{\alpha\beta}, B_\alpha)$. The first element $g_{\alpha\beta\gamma}$ is the same as
in the last section: on the triple intersection $U_{\alpha \beta \gamma} = U_\alpha
\cap U_\beta \cap U_\gamma$, the bundle gerbe multiplication induces an isomorphism
$\mu \maps L_{\alpha \beta} \otimes L_{\beta \gamma} \to L_{\alpha \gamma}$. We thus
have two sections of $L_{\alpha \gamma}$ over a triple intersection, namely
$\sigma_{\alpha \gamma}$ and $\mu(\sigma_{\alpha \beta} \otimes \sigma_{\beta
\gamma})$, both of which are trivializing. Thus we must have:
\[ \mu(\sigma_{\alpha \beta} \otimes \sigma_{\beta \gamma}) = g_{\alpha \beta \gamma}
  \sigma_{\alpha \gamma} , \]
for some even, invertible element of the complexified structure sheaf, $g_{\alpha
\beta \gamma} \in \Oh_\C^*(U_{\alpha \beta \gamma})$.

The even 1-form $A_{\alpha\beta} \in \Omega^1_{\C,0}(U_{\alpha\beta})$ on each double
intersection $U_{\alpha\beta}$ comes from our connection. Specifically, we pull
$\nabla$ back to $L_{\alpha\beta}$, and define $A_{\alpha \beta}$ to be the unique
even 1-form such that the following equation holds:
\[ \nabla(\sigma_{\alpha\beta}) = A_{\alpha\beta} \otimes \sigma_{\alpha\beta} . \]

Finally, the even 2-form $B_\alpha \in \Omega^2_{\C,0}(U_\alpha)$ on each open $U_\alpha$
is simply the pullback of the curving $B \in \Omega^2_{\C,0}(Y)$ along $s_\alpha$:
\[ B_\alpha = s_\alpha^* B . \]

The triple $(g_{\alpha\beta\gamma}, A_{\alpha\beta}, B_\alpha)$ is a degree 2 element
in the \Cech--Deligne total complex. With some work, we can use the definition of
bundle gerbe with connection to show that it is a 2-cocycle. With some more work, we
can show that the cohomology class of this 2-cocycle is independent of the choices we
have made. Thus we obtain a class $\Del(\G) \in H^2(M, \DelCpx(2))$ in Deligne
cohomology, the \define{Deligne class of the bundle gerbe with connection $\G$}.

The Deligne class just constructed is a complete invariant of a bundle gerbe with
connection, considered up to equivalence. We prove this in
Theorem~\ref{thm:deligne}. Its proof follows the proof of the analogous result over
smooth manifolds by Murray \cite{Murray} and Murray--Stevenson
\cite{MurrayStevenson}, adapted to the language of structure sheaves.

\begin{thm} \label{thm:deligne} Let $M$ be a supermanifold, $\mathbb{G}^\nabla(M)$
  the abelian group of equivalence classes of bundle gerbes with connection over $M$,
  and let $\Del \maps \mathbb{G}^\nabla(M) \to H^2(M,\DelCpx(2))$ be the map sending
  the equivalence class of a bundle gerbe with connection on $M$ to the corresponding
  Deligne cohomology class. Then:
  \begin{enumerate}
  \item $\Del \maps \mathbb{G}^\nabla(M) \to H^2(M, \DelCpx(2))$ is an isomorphism of
    abelian groups, i.e.,
    \begin{itemize}
    \item $\Del$ is onto;
    \item $\Del(\G) = 0$ if and only if $\G$ is trivial;
    \item $\Del(\G \otimes \H) = \Del(\G) + \Del(\H)$;
    \item $\Del(\G^*) = -\Del(\G)$;
    \end{itemize}
    where $\G$ and $\H$ are bundle gerbes with connection on $M$.
  \item If $f \maps M \to N$ is a map of supermanifolds, $\Del$ is natural with respect
    to pullback by $f$:
    \[ \Del(f^* \G) = f^* \Del(\G), \]
    for any bundle gerbe $\G$ with connection on $N$.
  \end{enumerate}
\end{thm}

\begin{proof}
  The proof of this theorem is similar to the proof of Theorem~\ref{thm:cech}, but
  differs when it comes to showing that $\Del$ is injective. We thus focus on this
  part. Because $\Del$ is a group homomorphism, it suffices to check that it has
  trivial kernel. In other words, that $\Del(\G) = 0$ implies $\G$ has a
  trivialization.

  So let $\G = (Y, L, \mu)$ be a bundle gerbe on $M$ with connection $(\nabla,
  B)$. Let $\Del(\G) = [g_{\alpha \beta \gamma}, A_{\alpha \beta}, B_\alpha]$ be the
  Deligne class constructed as above using a good cover $\{U_\alpha\}$, sections
  $s_\alpha \maps U_\alpha \to Y$, and trivializing sections $\sigma_{\alpha \beta}
  \in \L_{\alpha\beta}(U_{\alpha\beta})$. The equation $\Del(\G) = 0$ means our
  representative 2-cocycle $(g_{\alpha \beta \gamma}, A_{\alpha \beta}, B_\alpha)$ is
  a coboundary, which in turn means that
  \begin{eqnarray*}
    g_{\alpha \beta \gamma} & = & f_{\alpha\beta} f_{\beta\gamma}f_{\gamma\alpha}, \\
    A_{\alpha \beta} & = & z_\alpha - z_\beta + f_{\alpha\beta}^{-1} df_{\alpha\beta}^{\ns}, \\
    B_\alpha & = & dz_\alpha ,
  \end{eqnarray*}
  for some collection of even, invertible elements of the complexified structure
  sheaf, $f_{\alpha\beta} \in \Oh_\C^*(U_{\alpha \beta})$, and some collection of
  even 1-forms, $z_\alpha \in \Omega^1_{\C,0}(U_\alpha)$.

  Forgetting the connection $(\nabla, B)$ and the associated cocycle data
  $(A_{\alpha\beta}, B_\alpha)$, we are in precisely the situation of
  Theorem~\ref{thm:cech}. As in the proof of that theorem, we can construct a line
  bundle $T \to Y$ and a line bundle isomorphism $\tau \maps L \to \delta T$
  compatible with multiplication. To finish constructing a trivialization of $(\G,
  \nabla, B)$, we need to:
  \begin{itemize}
  \item Choose a connection $\nabla_T$ on $T$ such that, with the induced connection
    on $\delta T$, the map $\tau$ becomes an isomorphism of line bundles with
    connection;
  \item Check that $F_T = B$, for $F_T$ the curvature of $\nabla_T$, and $B$ the
    curving.
  \end{itemize}

  Finding the connection is straightforward: to begin, choose any connection
  $\nabla_T$ on $T$. Let $\delta \nabla_T$ denote the induced connection on $\delta
  T$. Using the isomorphism to identify $L$ and $\delta T$, the two line bundle
  connections differ by an even 1-form:
  \[ \nabla = \delta \nabla_T + a , \]
  for some $a \in \ECDR^1(Y^{[2]})$. The fact that $\nabla$ is compatible with
  multiplication implies $\delta a = 0$, and an application of Lemma~\ref{lem:murray}
  tells us that $a = \delta a'$ for some even 1-form $a' \in \ECDR^1(Y)$. Redefining
  $\nabla_T$ to be $\nabla_T - a'$, we have $\nabla = \delta \nabla_T$, as desired.

  Alas, with this choice $(T, \nabla_T)$ of line bundle and connection, we cannot
  guarantee that the curvature equals the curving, $F_T = B$. Instead, we can check
  that $F_T - B = \pi^* b$, for some 2-form $b$ on $M$, as follows: because
  $(\nabla, B)$ is a connection, it satisfies the descent equation, $F_\nabla =
  \delta B$. Moreover, because $\nabla = \delta \nabla_T$, we also have $F_\nabla =
  \delta F_T$. Combining these two facts, we see that $F_T - B$ is closed under
  $\delta$, so another application of Lemma~\ref{lem:murray} gives us the 2-form $b$.

  To finish constructing our trivialization, we note there is a line bundle with
  connection $(Q, \nabla_Q)$ on $M$ whose curvature is the 2-form $b$: $F_Q =
  b$. To see this, pull our line bundle $T$ back to the open set $U_\alpha$,
  $T_\alpha = s_\alpha^* T$, and choose trivializing sections $\sigma_\alpha \in
  \T_\alpha(U_\alpha)$. The connection $\nabla_T$ induces a connection
  $\nabla_{T_\alpha}$, and we let $a_\alpha \in \Omega^1_\C(U_\alpha)$ denote the
  corresponding connection 1-form:
  \[ \nabla_{T_\alpha} \sigma_\alpha = a_\alpha \otimes \sigma_\alpha . \]
  Finally, redefine the trivializing sections $\sigma_{\alpha\beta}$ of
  $L_{\alpha\beta}$ to be given by the difference $\sigma_{\alpha \beta} =
  \tfrac{\sigma_\alpha}{\sigma_\beta}$. With these choices, we have:
    \begin{eqnarray*}
    g_{\alpha \beta \gamma} & = & 1 , \\
    A_{\alpha \beta} & = & a_\alpha - a_\beta .
  \end{eqnarray*}
  Combining this with the fact that this 2-cocycle is a coboundary, we see:
  \begin{eqnarray*}
    1 & = & f_{\alpha\beta} f_{\beta\gamma}f_{\gamma\alpha}, \\
    a_\alpha - a_\beta & = & z_\alpha - z_\beta + f_{\alpha\beta}^{-1} df_{\alpha\beta}^{\ns}, \\
    B_\alpha & = & dz_\alpha ,
  \end{eqnarray*}
  Thus, there is a line bundle $(Q, \nabla_Q)$ on $M$ with descent data
  $(f_{\alpha\beta}, a_\alpha - z_\alpha)$. It is a quick exercise to check that $F_Q
  = b$ using the descent data.

  Now, with the line bundle $(Q, \nabla_Q)$ in hand, note that $T \otimes \pi^* Q^*$
  with its induced connection $\nabla_{T \otimes \pi^* Q^*}$ has curvature equal to
  the curving $B$. Moreover, because $\delta \pi^* = 0$, note that $\delta(T \otimes
  \pi^* Q^*) \iso \delta T$ and this isomorphism takes $\delta \nabla_{T \otimes
  \pi^* Q^*}$ to $\delta \nabla_T$. Thus we can redefine $T$ to be $T \otimes \pi^*
  Q^*$, and this yields our trivialization of $(\G, \nabla, B)$.
\end{proof}

We turn to the issue of when $\I_b \simeq \I_{b'}$. Let us write
$\mathbb{I}^\nabla(M)$ for equivalence classes of trivial bundle gerbes with
connection; more precisely, this is the set of all equivalence classes of bundle
gerbes with connection that have vanishing Dixmier--Douady class:
\[ \mathbb{I}^\nabla(M) = \left\{ [\G] \in \mathbb{G}^\nabla(M) \, : \, \DD(\G) = 0
  \right\} . \]
This is a group under tensor product, and we will now compute it. To facilitate this,
we need a definition: we say that a closed, complex-valued $p$-form $\alpha \in
\Omega^p_\C(M)$ is \define{integral} if its de Rham class $[\alpha]$ lies in the
image of the map $H^p(M_b, 2\pi i \Z) \to H^p(M_b, \C)$ from integral cohomology. 

\begin{prop}\label{prop:trivconn}
  The abelian group $\mathbb{I}^\nabla(M)$ of trivial bundle gerbes with connection
  is isomorphic to the group of even 2-forms modulo the group of even, closed,
  integral 2-forms, which we denote by $\Omega^2_\Z(M)$:
  \[ \mathbb{I}^\nabla(M) \iso \ECDR^2(M)/\Omega^2_\Z(M) . \]
\end{prop}
\begin{proof}
  Consider the map
  \[ \begin{array}{rcl}
       \ECDR^2(M) & \to & \mathbb{I}^\nabla(M) \\
       b & \mapsto & \I_b , 
     \end{array}
   \]  
   where $\I_b = (M, \Oh_\C, \mu_{\rm can})$ is the trivial bundle gerbe with curving
   $b \in \ECDR^2(M)$, as constructed in Example \ref{ex:trivconn}. Because the
   curvings add under tensor product, this map is a homomorphism. We wish to show
   that it is onto with kernel $\Omega^2_\Z(M)$.

   To show that it is onto, let $\G = (Y, L, \mu)$ be a bundle gerbe with connection
   $(\nabla, B)$ and vanishing Dixmier--Douady class. By Theorem~\ref{thm:cech},
   there is a line bundle $T \to Y$ and an isomorphism of line bundles $L \iso \delta
   T$ respecting bundle gerbe multiplication. We henceforth identify the line bundles
   $L$ and $\delta T$ using this isomorphism. By the same argument as in the proof of
   Theorem~\ref{thm:deligne}, we can equip $T$ with a connection $\nabla_T$ such that
   $\nabla = \delta \nabla_T$. Also by the same argument, there is a 2-form $b$ on
   the base such that:
   \[ B = F_T + \pi^* b . \]
   Thus $\G \simeq \I_b$.

   Finally. to show the kernel is $\Omega^2_\Z(M)$, we must check that $\I_b$ is
   trivializable if and only if $b$ is closed and integral. On the one hand, if $b$
   is closed and integral, then there is a line bundle $T \to M$ with connection and
   curvature $b$ \cite{Kostant}. This provides a trivialization. On the other hand,
   if there is a trivialization, $T \to M$, then $b = F_T$, so $b$ is closed and
   integral.
\end{proof}

\subsection{Curvature}
\label{sec:curvature}

To any connection on a bundle gerbe, we can associate a 3-form, called the
`curvature'. The curvature 3-form is closed, and we shall see it is a de Rham
representative of the Dixmier--Douady class.

\begin{defn}
  Let $\G = (Y, L, \mu)$ be a bundle gerbe with connection $(\nabla,B)$ over the
  supermanifold $M$. We define the \define{curvature $\curv(\G)$ of $\G$} as the
  unique complex-valued 3-form $\curv(\G) \in \Omega^3_\C(M)$ whose pullback to $Y$
  is the exterior derivative of the curving:
  \[ \pi^* \curv(\G) = dB . \]
\end{defn}

\noindent
The existence and uniqueness of this 3-form is a consequence of the descent equation,
\[ F_\nabla = \pi^*_1 B - \pi^*_2 B . \]
Indeed, recall that in the notation of Lemma \ref{lem:murray}, we write $\delta \maps
\Omega^2_\C(Y) \to \Omega^2_\C(Y^{[2]})$ for the map taking $B$ to $\pi^*_1 B -
\pi^*_2 B$. With this notation, the descent equation reads $F_\nabla = \delta
B$. Because the exterior derivative commutes with pullback, we quickly compute:
\[ \delta dB = d \delta B = d F_\nabla = 0 . \]
So, the 3-form $dB$ is $\delta$-closed, and Lemma \ref{lem:murray} guarantees a
unique 3-form $H$ on $M$ such that $dB = \pi^* H$. This 3-form $H$ is the curvature.

It is immediate from the definition that the curvature is closed, so it represents a
class in de Rham cohomology. The central fact about the curvature is that this class
is the same as the Dixmier--Douady class, up to a pesky sign:
\[ \DD(\G) = -[\curv(\G)] . \]
In order to make sense of this equation, we need to say how the Dixmier--Douady class
sits in de Rham cohomology. Recall the Dixmier--Douady class was constructed as a
degree-2 class in \Cech\ cohomology, $\DD(\G) \in \check{H}^2(M_b, \Oh_\C^*)$. We
then used a short exact sequence of sheaves:
\[ 0 \longrightarrow \Z_{M_b} \stackrel{2\pi i}{\longrightarrow} (\Oh_\C)_0
  \stackrel{\exp}{\longrightarrow} \Oh^*_\C \longrightarrow 0 , \]
to construct an isomorphism $\check{H}^2(M_b, \Oh^*_\C) \iso H^3(M_b, 2 \pi i
\Z)$. Here, the $2 \pi i$ arises naturally from its appearance in the short exact
sequence of sheaves. Of course, as abelian groups, $H^3(M_b, 2 \pi i \Z) \iso
H^3(M_b, \Z)$, but it is useful to keep the factor of $2 \pi i$ when comparing $\DD(\G)$
with the curvature.

Finally, because of the inclusion $2 \pi i \Z \inclusion \C$, we get a map $H^3(M_b,
2\pi i \Z) \to H^3(M_b, \C)$ from integral to complex coefficients. Abusing notation,
we continue to write $\DD(\G) \in H^3(M_b,\C)$ for the image of this map.

\begin{prop}
  \label{prop:rep}
  As cohomology classes in $H^3(M_b, \C)$, we have $\DD(\G) = -[\curv(\G)]$.
\end{prop}
\begin{proof}
  The core of the proof is a standard double complex argument such as one might find
  in Bott and Tu \cite{BottTu}, but we give it for completeness, following the proof
  in Waldorf's thesis \cite[Theorem 1.5.3]{WaldorfThesis}.

  We want to compare the Dixmier--Douady class $\DD(\G)$, constructed using \Cech\
  cohomology, with the de Rham class of the curvature $\curv(\G)$, and the natural
  way to do this is to work in the \Cech--de Rham double complex:
  \[ C^\bullet(\U, \CDR^\bullet) , \]
  for some good cover $\U = \{ U_\alpha \}$ of our supermanifold $M$. In
  supergeometry, we have both partitions of unity and the Poincar\'e lemma, so the
  usual arguments show that this complex has exact rows and columns. We take the
  total differential to be $D = \delta + (-1)^p d$, where $\delta$ is the \Cech\
  differential, $d$ is the exterior derivative, and $p$ is the \Cech\ degree.

  The result follows if we show that representatives of $\DD(\G)$ and $-\curv(\G)$ are
  cohomologous in the \Cech--de Rham complex, and we can do this starting with the
  Deligne class:
  \[ \Del(\G) = [g_{\alpha\beta\gamma}, A_{\alpha\beta}, B_{\alpha} ] . \]
  Because the cover is good, the exponential map $\exp \maps (\Oh_\C)_0(U_{\alpha
  \beta \gamma}) \to \Oh^*_\C(U_{\alpha \beta \gamma})$ is surjective \cite{Kostant},
  so we can choose a logarithm $h_{\alpha\beta\gamma} \in
  (\Oh_\C)_0(U_{\alpha\beta\gamma})$ for each $g_{\alpha\beta\gamma} \in
  \Oh^*_\C(U_{\alpha\beta\gamma})$.

  Having made this choice, the usual construction of the connecting homomorphism in
  \Cech\ cohomology tells us that $k_{\alpha\beta\gamma\delta} = (\delta
  h)_{\alpha\beta\gamma\delta}$ is a \Cech\ representative of the Dixmier--Douady
  class $\DD(\G)$, or more precisely of its image in $H^3(M_b, \C)$. Similarly,
  restricting the curvature $H = \curv(\G)$ to each open gives a family of 3-forms
  $H_\alpha = H|_{U_\alpha}$ that represents the de Rham class $[H]$ in the \Cech--de
  Rham complex. Now, a calculation shows:
  \[ D(h_{\alpha \beta \gamma}, A_{\alpha \beta}, B_\alpha) = (k_{\alpha \beta \gamma
    \delta}, 0, 0, 0, H_\alpha) . \]
  The right-hand side represents $DD(\G) + [H]$, and the equation says this is
  cohomologous to zero. This completes the proof.
  \end{proof}

\noindent
Because it lies in the image of $H^3(M_b, 2\pi i \Z) \to H^3(M_b, \C)$, the curvature
3-form is quite special: its de Rham class lies in a distinguished lattice inside the
larger vector space. In general, we say that a closed, complex-valued $p$-form
$\alpha \in \Omega^p_\C(M)$ is \define{integral} if its de Rham class $[\alpha]$ lies
in the image of the map $H^p(M_b, 2\pi i \Z) \to H^p(M_b, \C)$ from integral
cohomology.

Clearly, the curvature 3-form is integral. In fact, every even, closed, integral
3-form is the curvature 3-form of some bundle gerbe:

\begin{thm}
  \label{thm:integral}
  If $H \in \Omega^3_\C(M)$ is an even, closed, integral, complex-valued 3-form on
  the supermanifold $M$, then there is a bundle gerbe $\G$ with connection on $M$
  such that the curvature of $\G$ is $H$.
\end{thm}
\begin{proof}
  This proof is similar to that of Proposition \ref{prop:rep}, but running in
  reverse: in the \Cech--de Rham complex, we know that $H_\alpha = H|_{U_\alpha}$ is
  cohomologous to a \Cech\ representative $k_{\alpha\beta\gamma\delta} \in
  \Oh_\C(U_{\alpha\beta\gamma\delta})$ that, by hypothesis, is valued in $2\pi i
  \Z$. Thus:
  \[ D(h_{\alpha \beta \gamma}, A_{\alpha \beta}, B_\alpha) = (-k_{\alpha \beta \gamma
    \delta}, 0, 0, 0, H_\alpha) , \]
  for some \Cech--de Rham 2-cochain $(h_{\alpha\beta\gamma}, A_{\alpha\beta},
  B_\alpha)$. Setting $g_{\alpha \beta \gamma} = \exp(h_{\alpha\beta\gamma})$, a
  calculation shows that $(g_{\alpha \beta \gamma}, A_{\alpha \beta}, B_\alpha)$ is a
  2-cocycle in the \Cech--Deligne complex. Letting $\G$ be the bundle gerbe with this
  Deligne class, we have $\curv(\G) = H$.

\end{proof}

\subsection{Body and soul}
\label{sec:bodyandsoul}

We have noted that every supermanifold contains an ordinary manifold: the body. While
the inclusion of the body into the full supermanifold is canonical, there is no
canonical projection from the supermanifold down to the body. In fact, for
holomorphic supermanifolds, such a projection need not even exist.

We are working with smooth supermanifolds, so a projection does exist, but it is a
choice of extra data. In this section we fix such a choice. For a supermanifold $M$,
fix a \define{body projection}: a map $p \maps M \to M_b$ such that $pi = 1_{M_b}$,
where $i \maps M_b \inclusion M$ is the canonical inclusion of the body $M_b$.

In his treatment of supergeometry, it was DeWitt who christened the submanifold $M_b$
the `body' of $M$. He then called the directions in $M$ transverse to the body $M_b$
the `soul' \cite{DeWitt}. Although our framework for supergeometry is based on
locally ringed spaces and is distinct from DeWitt's formalism, we adopt his poetic
terminology.

For any supermanifold $M$, pullback along the inclusion $i \maps M_b \inclusion M$ gives us
a short exact sequence:
\[ 0 \longrightarrow \Soul^\bullet(M) \longrightarrow \Omega^\bullet(M)
  \stackrel{i^*}{\longrightarrow} \Omega^\bullet(M_b) \longrightarrow 0 . \]
Here, $\Omega^\bullet(M_b)$ is the ordinary de Rham complex on the ordinary manifold
$M_b$; we call it the \define{body of the de Rham complex $\Omega^\bullet(M)$}. The
kernel of the pullback, $\Soul^\bullet(M) = \ker(i^*)$, is the subcomplex of forms
that go to zero upon restriction to the body. We call this subcomplex
$\Soul^\bullet(M)$ the \define{soul of the de Rham complex $\Omega^\bullet(M)$}.

Having fixed a body projection $p \maps M \to M_b$, we get a splitting of this short
exact sequence.

\begin{prop}
  Let $M$ be a supermanifold equipped with a body projection $p \maps M \to
  M_b$. Then the pullback along the body projection, $p^*$, splits the short exact
  sequence of forms:
  \[
    \begin{tikzcd}
      0 \ar[r] & \Soul^\bullet(M) \ar[r] & \Omega^\bullet(M) \ar[r, "i^*"] & \Omega^\bullet(M_b) \ar[r] \ar[l, bend left=30, "p^*"] & 0 
    \end{tikzcd}
  \]
  % \[ 0 \longrightarrow \Soul^\bullet(M) \longrightarrow \Omega^\bullet(M)
  %   \stackrel{i^*}{\longrightarrow} \Omega^\bullet(M_b) \longrightarrow 0 . \]
  In particular, the de Rham complex $\Omega^\bullet(M)$ decomposes into a direct sum
  of body and soul:
  \[ \Omega^\bullet(M) \iso \Omega^\bullet(M_b) \oplus \Soul^\bullet(M) , \]
  where the projection map $\Omega^\bullet(M) \to \Omega^\bullet(M_b)$ is $i^*$ and
  the inclusion map $\Omega^\bullet(M_b) \inclusion \Omega^\bullet(M)$ is $p^*$.
\end{prop}
\begin{proof}
  To show that $i^*$ is onto, first note that this is automatic in any coordinate
  patch with coordinates $(x_i, \theta_j)$: in this patch, a form on the body is just
  a form on the supermanifold independent of the $\theta_j$ coordinates. Surjectivity
  now follows from a partition of unity argument.

  To show that the sequence splits, note that $pi = 1$ by the definition of the body
  projection.  Thus $i^* p^* = 1$, and we are done.
\end{proof}

\noindent
Using this proposition, we see that any $k$-form $\omega \in \Omega^k(M)$ uniquely
decomposes into its \define{body} $\omega_b = p^* i^* \omega$ and its \define{soul}
$\omega_s = \omega - p^* i^* \omega$. The importance of the body and soul
decomposition lies in the fact that the soul of the de Rham complex
$\Soul^\bullet(M)$ is acyclic: its cohomology is trivial in all degrees.

\begin{prop}
  \label{prop:acyclic}
  $H^\bullet(\Soul^\bullet(M)) = 0$. In other words, if $\omega = \omega_s$ is a
  closed $k$-form that is pure soul, then it is exact: $\omega = d\sigma$, for some
  $\sigma \in \Soul^{k-1}(M)$.
\end{prop}
\begin{proof}
  In his notes on supergeometry~\cite{Kostant}, Kostant observes that restriction to
  the body, $i^* \maps \Omega^\bullet(M) \to \Omega^\bullet(M_b)$, is a
  quasi-isomorphism: it induces isomorphism on cohomology. So, the map $i^*$ is a
  surjective map of complexes that is also a quasi-isomorphism, and it follows that
  its kernel, $\Soul^\bullet(M)$, is an acyclic complex.
\end{proof}

\noindent
Now let us apply the body and soul decomposition to the curvature of a bundle
gerbe. If $H = \curv(\G)$ is the curvature 3-form of a bundle gerbe with connection,
then applying the body and soul decomposition to the complexified de Rham complex, we
have $H = H_b + H_s$. Both $H_b$ and $H_s$ are closed, so by the last proposition
there is a pure soul 2-form $\beta \in \Soul^2(M)$ such that $H_s = d\beta$, and we
find:
\[ H = H_b + d\beta . \]
We noted in the last section that $H$ is integral: its cohomology class $[H] \in
H^3(M_b, \C)$ lies in the image of integral cohomology, $H^3(M_b, 2\pi i \Z) \to
H^3(M_b, \C)$. Because $[H] = [H_b]$, the body $H_b$ must be integral as well.

Thus, from the curvature 3-form $H$ of a bundle gerbe $\G$ on a supermanifold $M$, we
have constructed an ordinary, closed, integral 3-form $H_b$ on the body $M_b$. From
the theory of bundle gerbes, it now follows that there is an ordinary bundle gerbe
$\G_b$ on the body $M_b$ with curvature $H_b = \curv(\G_b)$. By ``ordinary'', we mean
that this bundle gerbe is constructed entirely in the category of smooth manifolds:
$\G_b = (Y_b, L_b, \mu_b)$ for $Y_b$ a smooth manifold and $Y_b \to M_b$ a surjective
submersion, $L_b$ a line bundle on $Y_b^{[2]}$, and so forth.

How does the bundle gerbe $\G_b$ in the category of smooth manifolds compare to the
bundle gerbe $\G$ that we started with in the category of supermanifolds? Using the
projection $p \maps M \to M_b$, we can try to compare them: the pullback $p^* \G_b$
is a bundle gerbe on $M$. This new bundle gerbe on $M$ has different curvature from
$\G$, but only just: $\curv(p^* \G_b) = H_b$, while $\curv(\G) = H_b + d\beta$. This
suggests that if we add $\beta$ to the curving of $p^* \G_b$, we might recover
$\G$. Indeed:
\[ \curv(p^* \G_b \otimes \I_\beta ) = H . \]
Evan so, $\G$ need not be equivalent to $p^* \G \otimes \I_\beta$. As a pair of
bundle gerbes on $M$ with the same curvature, they can differ by a flat bundle gerbe.

To understand this, fix a 3-form $H$ and consider the set of bundle gerbes with
curvature $H$:
\[ \Grb{M}{H} = \left\{ [\G] \in \mathbb{G}^\nabla(M) \, : \, \curv(\G) = H \right\}
  . \]
Of course, this set is empty unless $H$ is even, closed and integral, so we assume
this from now on. This subset of the group $\mathbb{G}^\nabla(M)$ is not itself a
group unless $H = 0$, because curvatures add under tensor product: $\curv( \G \otimes
\H) = \curv(\G) + \curv(\H)$.

We call a bundle gerbe $\G$ \define{flat} if $\curv(\G) = 0$. The group $\Grb{M}{0}$
of equivalence classes of flat bundle gerbes is an essential tool for understanding
how two bundle gerbes with the same curvature can differ. This is because, for any
two bundle gerbes $\G$ and $\H$ with the same curvature $H$, there is a unique flat
bundle gerbe $\G_0$ such that:
\[ \G \otimes \G_0 \simeq \H .  \]
Indeed, tensoring both sides with $\G^*$, we must have $\G_0 \simeq \G^* \otimes
\H$. We have proved:
\begin{prop}
  For any even, closed, integral 3-form $H \in \CDR^3(M)$, the set $\Grb{M}{H}$ of
  equivalence classes of bundle gerbes with curvature $H$ is a torsor for the group
  of flat bundle gerbes, $\Grb{M}{0}$.
\end{prop}

\noindent
Returning to the body and soul decomposition, we would like to compare the torsor of
bundle gerbes of specified curvature on the body with that on the full
supermanifold. To begin, let us compare the group of flat bundle gerbes
$\Grb{M_b}{0}$ on the body with the that on the full supermanifold,
$\Grb{M}{0}$. Using our fixed projection map $p \maps M \to M_b$, we can pullback
flat bundle gerbes on the body. The resulting bundle gerbe on $M$ is still flat, and
this is an isomorphism:
\begin{prop}
  Let $M$ be a supermanifold equipped with a body projection $p \maps M \to
  M_b$. Pullback of flat bundle gerbes along the body projection induces an
  isomorphism:
  \[ p^* \maps \Grb{M_b}{0} \to \Grb{M}{0} \]
  between the group of flat bundle gerbes on the body $M_b$ and on the the full
  supermanifold $M$.
\end{prop}
\begin{proof}
  To show $p^*$ has trivial kernel, suppose $p^* \G \simeq \I$. Because $pi = 1$ for
  the canonical inclusion $i$, applying $i^*$ yields $\G \simeq \I$. On the other
  hand, let $\G$ be a flat bundle gerbe on the supermanifold $M$ and consider the
  bundle gerbe $\G_b = i^* \G$ on the body. We claim $p^* \G_b \simeq \G$. Indeed,
  since $p$ and $i$ act trivially on the Dixmier--Douady class, we have that $\DD(p^*
  \G_b \otimes \G^*) = 0$. Thus from Proposition \ref{prop:trivconn}, we have:
  \[ p^* \G_b \otimes \G^* \simeq \I_\beta \]
  for some 2-form $\beta$. Moreover, we can choose $\beta$ to be pure soul, because
  applying $i^*$ kills it. Now, since the left-hand side is flat, we must have
  $d\beta = 0$. But a closed pure soul 2-form is integral, so $\I_\beta \simeq
  \I$. We conclude $p^* \G_b \simeq \G$, as desired.
\end{proof}

The proof of the previous proposition actually shows that restriction to the body
$i^*$ is the inverse of $p^*$, so it turns out the isomorphism between the groups of
flat bundle gerbes is independent of the choice of body projection. Nevertheless, the
choice of body projection is essential for comparing the torsors $\Grb{M_b}{H_b}$ and
$\Grb{M}{H}$, so we have chosen to emphasize it.

\begin{thm}
  \label{thm:torsors}
  Let $M$ be a supermanifold equipped with a body projection $p \maps M \to M_b$, and
  let $H \in \CDR^3(M)$ be an even. closed, integral 3-form. Choosing an even 2-form
  $\beta \in \CDR^2(M)$ such that $H_s = d\beta$ is the soul of $H$, we have the
  following isomorphism of $\Grb{M_b}{0}$-torsors:
  \[
    \begin{array}{ccc}
      \theta \maps \Grb{M_b}{H_b} & \to & \Grb{M}{H}, \\
      \G_b & \mapsto & p^* \G_b \otimes \I_\beta .
    \end{array}
  \]
\end{thm}
\begin{proof}
  Because any map between torsors is an isomorphism, we need only show that $\theta$
  intertwines the action. For any bundle gerbe $\G_b$ on the body with curvature
  $H_b$, and any flat bundle gerbe $\G_0$, we have:
  \[ \theta( \G_b \otimes \G_0 ) = p^*(\G_b \otimes \G_0) \otimes \I_\beta \simeq
    \left( p^* \G_b \otimes \I_\beta \right) \otimes p^* \G_0 , \]
  using the fact that pullback respects the tensor product. Since we are using $p^*$
  to identify the groups of flat bundle gerbes, this is the desired result.
\end{proof}

\noindent
As a corollary, we get the body and soul decomposition of any bundle gerbe.

\begin{cor}
  Let $\G$ be a bundle gerbe on a supermanifold $M$, and let $p \maps M \to M_b$ be a
  body projection. Then there is a bundle gerbe $\G_b$ on the body $M_b$ and a pure
  soul 2-form $\beta$ such that:
  \[ \G \simeq p^* \G_b \otimes \I_\beta . \]
  Moreover, $\G_b \simeq i^* \G$, and this is the unique choice up to equivalence,
  while $\beta$ is unique up to the addition of an exact pure soul 2-form.
\end{cor}

\begin{proof}
  The existence of this decomposition follows from Theorem \ref{thm:torsors} with $H
  = \curv(\G)$. Applying $i^*$ to both sides of $\G \simeq p^* \G_b \otimes
  \I_\beta$, see that $\G_\b \simeq i^* \G$, where we have used $i^* \I_\beta \simeq \I$,
  because $\beta$ is pure soul. Finally, if
  \[ p^* \G_\b \otimes \I_\beta \simeq p^* \G_b \otimes \I_{\beta'} \]
  for two pure soul 2-forms $\beta$ and $\beta'$, then we can tensor on both sides
  with $p^* \G_\b^*$ to conclude $\I_\beta \simeq \I_{\beta'}$. From Proposition
  \ref{prop:trivconn}, we conclude that $\beta - \beta'$ is closed, and Proposition
  \ref{prop:acyclic} tells us it is exact.
\end{proof}

\end{document}